\documentclass{amsart}
\usepackage[margin=1.0in]{geometry}
\usepackage{amsthm, amsmath, amssymb, mathtools, mathrsfs}
\usepackage{trimclip, enumitem, adjustbox}
\usepackage{tikz}
\usepackage{tikz-qtree}

\usepackage[bbgreekl]{mathbbol}
\usepackage{bbm}

\usepackage{aliascnt}
\usepackage{hyperref}
\usepackage{breakurl}
\usepackage{array, longtable}

\makeatletter
\def\LT@start{%
\let\LT@start\endgraf
\endgraf\penalty\z@\vskip\LTpre
\dimen@\pagetotal
\advance\dimen@ \ht\ifvoid\LT@firsthead\LT@head\else\LT@firsthead\fi
\advance\dimen@ \dp\ifvoid\LT@firsthead\LT@head\else\LT@firsthead\fi
\advance\dimen@ \ht\LT@foot
\dimen@ii\vfuzz
\vfuzz\maxdimen
\setbox\tw@\copy\z@
\setbox\tw@\vsplit\tw@ to \ht\@arstrutbox
\setbox\tw@\vbox{\unvbox\tw@}%
\vfuzz\dimen@ii
\advance\dimen@ \ht
\ifdim\ht\@arstrutbox>\ht\tw@\@arstrutbox\else\tw@\fi
\advance\dimen@\dp
\ifdim\dp\@arstrutbox>\dp\tw@\@arstrutbox\else\tw@\fi
\advance\dimen@ -\pagegoal
\ifdim \dimen@>\z@\unskip\vfil\break\fi
\global\@colroom\@colht
\ifvoid\LT@foot\else
\advance\vsize-\ht\LT@foot
\global\advance\@colroom-\ht\LT@foot
\dimen@\pagegoal\advance\dimen@-\ht\LT@foot\pagegoal\dimen@
\maxdepth\z@
\fi
\ifvoid\LT@firsthead\copy\LT@head\else\box\LT@firsthead\fi
\output{\LT@output}}
\makeatother

\emergencystretch=\maxdimen

\usepackage[textsize=footnotesize]{todonotes}

\usepackage{xcolor}
\definecolor{dblue}{rgb}{0,0,0.70}
\hypersetup{
	unicode=true,
	colorlinks=true,
	citecolor=dblue,
	linkcolor=dblue,
	anchorcolor=dblue
}

\makeatletter
\expandafter\g@addto@macro\csname th@plain\endcsname{%
	\thm@notefont{\bfseries}
}%
\expandafter\g@addto@macro\csname th@remark\endcsname{%
	\thm@headfont{\bfseries}
}%
\makeatother

\newtheorem{theorem}{Theorem}[section]	
\newtheorem*{theorem*}{Theorem}

\newaliascnt{lemma}{theorem}
\newtheorem{lemma}[lemma]{Lemma}
\aliascntresetthe{lemma}
\newtheorem*{lemma*}{Lemma}

\newaliascnt{proposition}{theorem}
\newtheorem{proposition}[proposition]{Proposition}
\aliascntresetthe{proposition}

\newaliascnt{corollary}{theorem}
\newtheorem{corollary}[corollary]{Corollary}
\aliascntresetthe{corollary}

\theoremstyle{remark}

\newaliascnt{remark}{theorem}

\aliascntresetthe{remark}
\newaliascnt{question}{theorem}
\newtheorem{question}[question]{Question}
\aliascntresetthe{question}

\newtheorem*{question*}{Question}

\newaliascnt{definition}{theorem}
\newtheorem{definition}[definition]{Definition}
\aliascntresetthe{definition}

\newaliascnt{example}{theorem}
\newtheorem{example}[example]{Example}
\aliascntresetthe{example}

\newaliascnt{convention}{theorem}

\aliascntresetthe{convention}

\makeatletter
\def\l@subsection{\@tocline{2}{0pt}{1pc}{5pc}{}} \def\l@subsection{\@tocline{2}{0pt}{2pc}{6pc}{}}
\makeatother
\usepackage[style = numeric, sorting = nyt]{biblatex}
\title{The Axiom of Double Complement and its opposites}
\author{Hanul Jeon}
\email{ \href{mailto:hanuljeon95@gmail.com}{hanuljeon95@gmail.com}}
\urladdr{ \href{https://hanuljeon95.github.io}{https://hanuljeon95.github.io} }
\address{Institute of Discrete Mathematics and Geometry, Vienna University of Technology. Wiedner Hauptstra\ss e 8–10, 1040 Vienna, Austria} 

\author{James E. Hanson}
\email{ \href{mailto:jameseh@iastate.edu}{jameseh@iastate.edu}}
\urladdr{ \href{https://james-hanson.github.io/}{https://james-hanson.github.io} }
\address{Department of Mathematics, Iowa State University, Ames, IA 50011} 

\thanks{We appreciate Thomas Forster for suggesting the Kripke models as a tool to approach the problem. We are also grateful to Andrew W. Swan for suggesting proof of \autoref{Theorem:PreservationDComOrig} under classical logic. The research of the first author is supported in part by NSF grant DMS--2153975 and DMS--2451350.}

\renewcommand{\mathbb}{\mathbbm}

\newcommand{\dcom}{{\lnot\lnot}}

\mathchardef\mhyphen="2D
\newcommand{\restricts}{\mkern-5mu\upharpoonright}
\newcommand{\VDash}{ %
\mathrel{\text{\clipbox{0pt 0pt {.8\width} 0pt}{$\Vdash$}}\mkern.9mu}\vDash}

\newcommand{\rrarrows}{\rightrightarrows}
\newcommand{\llarrows}{\leftleftarrows}
\newcommand{\lrlrarrows}{\mathrel{\mathrlap{\leftleftarrows}{\rightrightarrows}}}

\makeatletter
\DeclareRobustCommand\widecheck[1]{{\mathpalette\@widecheck{#1}}}
\def\@widecheck#1#2{%
    \setbox\z@\hbox{\m@th$#1#2$}%
    \setbox\tw@\hbox{\m@th$#1%
       \widehat{%
          \vrule\@width\z@\@height\ht\z@
          \vrule\@height\z@\@width\wd\z@}$}%
    \dp\tw@-\ht\z@
    \@tempdima\ht\z@ \advance\@tempdima2\ht\tw@ \divide\@tempdima\thr@@
    \setbox\tw@\hbox{%
       \raise\@tempdima\hbox{\scalebox{1}[-1]{\lower\@tempdima\box
\tw@}}}%
    {\ooalign{\box\tw@ \cr \box\z@}}}
\makeatother

\newcommand{\mv}{\operatorname{mv}}
\newcommand{\dom}{\operatorname{dom}}
\newcommand{\ext}{\operatorname{ext}}
\newcommand{\ran}{\operatorname{ran}}
\newcommand{\Dec}{\operatorname{Dec}}
\newcommand{\WDec}{\operatorname{WDec}}
\newcommand{\Ord}{\mathrm{Ord}}
\newcommand{\up}{\operatorname{up}}
\newcommand{\op}{\operatorname{op}}
\newcommand{\Union}{\operatorname{Union}}
\newcommand{\Power}{\operatorname{Power}}
\newcommand{\Sep}{\operatorname{Sep}}
\newcommand{\krk}{\operatorname{krk}}

\begin{document}
\maketitle

\begin{abstract}
	Powell introduced the Axiom of Double Complement ($\mathsf{DCom}$) to give his double-negation interpretation of $\mathsf{ZF}$ into $\mathsf{IZF_{Rep}}$. However, the consistency, strength, and compatibility of $\mathsf{DCom}$ remain open problems. This article aims to survey the compatibility and consistency strength of $\mathsf{DCom}$, its consequence and opposites, which will be named $\mathsf{NDCom}$ and $\mathsf{ADCom}$. We will also develop Lubarsky's Kripke models over $\mathsf{CZF}$ to derive these results.

	We will show that $\mathsf{DCom}$ proves the Powerset axiom over $\mathsf{CZF}$ and is independent of $\mathsf{IZF}$. We will also show that $\mathsf{ADCom}$ does not add consistency strength over $\mathsf{CZF}$, by modifying the construction of Lubarsky's model for $\mathsf{CZF+\lnot Pow}$. We will also show that $\mathsf{DCom}$, $\mathsf{ADCom}$, and $\mathsf{NDCom}$ are persistent under realizability under modest conditions.
\end{abstract}

\tableofcontents

\section{Introduction}
Powell \cite{Powell1975} introduced the Axiom of Double Complement ($\mathsf{DCom}$), which states that every set $x$ has a \emph{double complement}
\begin{equation*}
    x^\dcom = \{z : \lnot\lnot (z\in x)\}.
\end{equation*} 
Powell used $\mathsf{DCom}$ to construct an inner model of $\mathsf{ZF}$ under $\mathsf{IZF_{Rep}+DCom}$, intuitionistic $\mathsf{ZF}$ with Replacement in place of Collection and $\mathsf{DCom}$.
However, Powell's consistency proof is not widely known, and apparently not accepted into the mathematical community.
One suspected reason is that the relationship between $\mathsf{IZF}$ and $\mathsf{DCom}$ is unclear, even though Powell's method has attractive features. For example, it does not involve any non-extensional set theories appearing in H. Friedman's double-negation translation  \cite{Friedman1973}.
Some literature studied $\mathsf{DCom}$, e.g., Grayson \cite{Grayson1979} showed that Powell's inner model is isomorphic to the forcing extension $V^{\mathcal{P}^{\lnot\lnot}(1)}$ under $\mathsf{IZF+DCom}$. Vladimirov \cite{Vladimirov2018} investigated absoluteness of certain arithmetical formulas over $\mathsf{IZF+DCom}$ and its extensions. 

Hahanyan's works would be most prominent among studies about the Axiom of Double Complement. 
He examined the relationship between $\mathsf{DCom}$, other axioms of $\mathsf{IZF}$, and non-classical axioms. For example, he proved in \cite{Hahanyan1981} that $\mathsf{IZF+DCom}$ is consistent with some Brouwnian principles and Church's thesis.%
\footnote{We need to notice that Hahanyan uses formal systems that are different from what is usually called $\mathsf{IZF}$. For example. Hahanyan's $\mathsf{IZF}$ takes natural numbers and real numbers as urelements, not sets. However, methods in \cite[Chapter VII, \S1]{Beeson1985} seem to provide a way to translate Hahanyan's result to results for the standard $\mathsf{IZF}$.}

Hahanyan \cite{Hahanyan1992} claimed in 1992 that he proved $\mathsf{DCom}$ is independent of $\mathsf{IZF}$, but the author fails to find a full proof of this result. Six years later, Hahanyan \cite{Hahanyan1998} published a proof of a weakening of the previous independence result: he proved that $\mathsf{DCom}$ is not provable from $\mathsf{IZF- Pow}$, Intuitionistic set theory without Powerset.
Unfortunately, Hahayan's method does not seem to be adequate to derive the independence of $\mathsf{DCom}$ over $\mathsf{IZF}$. His proof employs functional realizability, and functional realizability satisfies the negation of $\mathsf{Pow}$. In fact, functional realizability satisfies the Axiom of Anti-Double Complement $\mathsf{ADCom}$ as we will see in \autoref{Section: FunctionalRealizability}, and $\mathsf{ADCom}$ implies the negation of Powerset.

The purpose of this paper is to clarify the relation between $\mathsf{DCom}$ and constructive set theories like $\mathsf{CZF}$ and $\mathsf{IZF}$.
Here is a brief description of each section:
\begin{itemize}
	\item In \autoref{Section:Prelim}, we will introduce some preliminaries, including $\mathsf{IZF}$ and $\mathsf{CZF}$, some variations of power set operations, and a fully inductive relation, a stronger version of a well-founded relation.
	
	\item In \autoref{Section:DCom}, we will examine basic properties of $\mathsf{DCom}$ and  stronger forms of its negation called $\mathsf{NDCom}$ and $\mathsf{ADCom}$. We will also examine double complements of some simple sets like $2$ and $\omega$. Interestingly, the double complement of $2$ is $\mathcal{P}(1)$. Hence $\mathsf{DCom}$ implies the Axiom of Power Set $\mathsf{Pow}$ over $\mathsf{CZF}$. We also discuss the consistency strength of $\mathsf{DCom}$ and $\mathsf{NDCom}$ over $\mathsf{CZF}$.
	
	\item In \autoref{Section:KripkeModel}, we will develop the theory of Kripke sets inspired by Lubarsky \cite{Lubarsky2005, HendtlassLubarsky2016}. We want to work over $\mathsf{CZF}$, so we will make his Kripke models fit on $\mathsf{CZF}$.
	
	\item In \autoref{Section:DComAndKripke}, we will derive some results on $\mathsf{DCom}$ by employing Kripke models. We will show that the Kripke universe $V^\mathbb{P}$ satisfies $\mathsf{DCom}$ if $V$ is a model of $\mathsf{ZF}$ and $\mathbb{P}$ is linear.
	We will also see that our proof could not work well when $\mathbb{P}$ is not linear. We will see that there is a Kripke frame $\mathbb{P}$ such that $V^{\mathbb{P}}$ satisfies $\mathsf{NDCom}$, even when we start with $V\vDash \mathsf{ZFC}$.
	
	\item In \autoref{Section:ADCom}, we will construct a Kripke model of $\mathsf{CZF+ADCom}$ by modifying the construction of Lubarsky's first model for $\mathsf{CZF+\lnot Pow}$ in \cite{Lubarsky2005}. The whole construction will be performed over $\mathsf{CZF}$, so it also establishes the equiconsistency between $\mathsf{CZF}$ and $\mathsf{CZF+ADCom}$.
	
	\item In \autoref{Section:MRrealizability}, we show that $\mathsf{DCom}$, $\mathsf{NDCom}$, and $\mathsf{ADCom}$ are persistent under realizability (see \cite{McCartyPhD, Rathjen2003Realizability} for details of realizability). We need $\Sigma_1$-separation or the regular extension axiom to establish the persistence of $\mathsf{DCom}$, and $\mathsf{Pow}$ for the persistence of $\mathsf{NDCom}$.
	
	\item In \autoref{Section: FunctionalRealizability}, we define functional realizability that resembles the important feature of Hahanyan's model in \cite{Hahanyan1998} and show that functional realizability satisfies $\mathsf{ADCom}$.
\end{itemize}

\section{Preliminaries}
\label{Section:Prelim}
\subsection{Intuitionistic and Construcive Set theories}
$\mathsf{ZF}$ provides a satisfactory formulation of classical set theory. However, there are nonequivalent constructive formulations of set theory corresponding to $\mathsf{ZF}$.

For example, Intuitionistic ZF ($\mathsf{IZF}$), which is obtained by replacing Replacement with Collection, Regularity with $\in$-induction, and uses intuitionistic logic instead of classical logic, is an example of set theory in constructive mathematics.

Another example is Constructive ZF ($\mathsf{CZF}$) introduced by Aczel (see \cite{Aczel1978, Aczel1982, Aczel1986}). $\mathsf{CZF}$ is obtained by replacing Collection with Strong Collection, and Powerset with Subset Collection from axioms of $\mathsf{IZF}$. 
Unlike $\mathsf{IZF}$, $\mathsf{CZF}$ allows type-theoretic interpretation. Moreover, $\mathsf{CZF}$ is equiconsistent with Kripke-Platek set theory $\mathsf{KP}$ and its intuitionistic version $\mathsf{IKP}$, which are dramatically weaker than $\mathsf{IZF}$. On the other hand, $\mathsf{IZF}$ is equiconsistent with $\mathsf{ZF}$. Despite their different consistency strength, both theories coincide if we add the law of excluded middle ($\mathsf{LEM}$) to each theory.

We will see some variations of $\mathsf{IZF}$ and $\mathsf{CZF}$ in this article.  $\mathsf{IZ}$ and $\mathsf{CZ}$ are theories obtained by dropping Collection or Strong Collection from $\mathsf{IZF}$ and $\mathsf{CZF}$ respectively. Similarly, $\mathsf{IZF}^-$ and $\mathsf{CZF}^-$ are theories that are obtained by forgetting Powerset or Subset Collection from $\mathsf{IZF}$ and $\mathsf{CZF}$, respectively.

We do not make an effort to examine the details of intuitionistic and constructive set theory. Readers could consult with \cite{AczelRathjen2010} or \cite[Chapter 2]{ZieglerPhD} for basic facts on constructive set theory if needed. However, we will formulate and explain some notions to clarify their meaning.

Various notions over $\mathsf{CZF}$ are obtained by replacing functions with \emph{multi-valued functions}, which are relations with the fixed domain.
\begin{definition}
    A relation $R\subseteq A\times B$ is a \emph{multi-valued function from $A$ to $B$} if the $\dom R=A$ and $\ran R\subseteq B$.
    We use the notation $R\colon A\rrarrows B$ for this case.
    We use the reversed symbol $R\colon A\llarrows B$ if $\dom R\subseteq A$ and $\ran R=B$.
    If both of $R\colon A\rrarrows B$ and $R\colon A\llarrows B$ hold, then we write $R\colon A\lrlrarrows B$.
\end{definition}

Every function has a corresponding image. Similarly, we can think of an analogous notion called \emph{subimage}:
\begin{definition}
    Let $R\colon A\rrarrows B$ be a multi-valued function.
    A class $C$ is a \emph{subimage} of $R$ if $R\colon A\lrlrarrows C$.
\end{definition}

We can see that if $R$ is a function, then a subimage is unique and equals the image. However, a subimage of a given multi-valued function is not unique in general:
\begin{example}
    Let $A=B=3=\{0,1,2\}$ and consider
    \begin{equation*}
        R = \{\langle 0,0\rangle, \langle 0,1\rangle, \langle 1,1\rangle, \langle 2,2\rangle\}.
    \end{equation*}
    It is easy to see that both of $\{1,2\}$ and $\{0,1,2\}$ are subimages of $R$.
\end{example}

Various notions of $\mathsf{CZF}$ are obtained by replacing functions with multi-valued functions, and images with subimages.
\begin{example}
    Replacement states the following claim: for every set $a$ and a class function $F\colon a\to V$, we have an image of $F$.
    
    By substituting a function with a multi-valued function, and an image with a subimage, we have the statement of Strong Collection: for every set $a$ and a class \emph{multi-valued function} $R\colon a\rrarrows V$, we have a \emph{subimage} of $R$.
\end{example}

The case for Subset Collection requires additional elaboration. 
First, let us consider the following form of power set, which may not be obvious at first glance:
\begin{lemma}[$\mathsf{ZF}^-$]
    The Powerset is equivalent to the following claim: for every set $a$ and $b$, we can find a set $c$ such that if $f\colon a\to b$ is a function, then the image of $f$ is a member of $c$.
    
    Especially, $\{\ran f\mid f\colon a\to b\}$ exists.
\end{lemma}
\begin{proof}
    Clearly, Powerset implies the above claim. Conversely, for given set $a$, consider $c=\{\ran f \mid f\colon a\to a\times 2\}$. Then we have
    \begin{equation*}
        \{g\in c \mid \forall x\in a\exists i\in 2 (\langle x,i\rangle\in g)\}
    \end{equation*}
    is the set of all functions from $a$ to $2$. Hence, Powerset holds.
\end{proof}

Now replace the terminologies appropriately so that we have the following assertion:
for every set $a$ and $b$, we can find a set $c$ such that if $r\colon a\rrarrows b$ is a multi-valued function from $a$ to $b$, then we can find a subimage $d\in c$ of $r$. Let us introduce the following terminology to describe the property of $c$ we get:
\begin{definition}
    Let $a$, $b$, and $c$ be sets. We say $c$ is \emph{full in subimages of multi-valued functions from $a$ to $b$} if $r\colon a\rrarrows b$, then we can find a subimage $d\in c$ of $r$.
\end{definition}

This statement itself is not the very statement of Subset Collection, but it turns out that this statement is equivalent to Subset Collection:
\begin{proposition}[$\mathsf{CZF}^-$]\label{Proposition:SubsetCollectionEquivalentFormulations}
    The following statements are equivalent:
    \begin{enumerate}
        \item Subset Collection: for every class family of relations $R_u\subseteq a\times b$ parameterized by $u$, we can find a set $c$ such that if $R_u\colon a\rrarrows b$, we can find $d\in c$ such that $d$ is a subimage of $R_u$.
        \item For every set $a$ and $b$, we can find a set $c$ which is full in subimages of multi-valued functions from $a$ to $b$.
    \end{enumerate}
\end{proposition}

Before going to the proof, let us examine the following lemma, which turns out to be useful for working on multi-valued functions.
\begin{lemma}\label{Lemma:PrelimAdjuectmentFtn}
    Let $R: A\rightrightarrows B$, $R\subseteq A\times B$ be a multi-valued function. Define $\mathcal{A}^{A,B}(R) = \{\langle a,\langle a,b\rangle\rangle \in A\times (A\times B) \mid \langle a,b\rangle \in R\}$. Then the following holds:
    \begin{enumerate}
        \item $\mathcal{A}^{A,B}(R) \colon A\rightrightarrows S\iff R\cap S:A\rightrightarrows B$,
        \item $\mathcal{A}^{A,B}(R) \colon A\leftleftarrows S\iff S\subseteq R$.
    \end{enumerate}
\end{lemma}
\begin{proof}
    For notational convenience, let us omit superscripts from $\mathcal{A}^{A,B}$ and write $\mathcal{A}$.
    We can see that $\mathcal{A}(R)\colon A\rightrightarrows S$ is equivalent to
    \begin{equation*}
    \forall a\in A \exists s\in S \bigl(\langle a,s\rangle \in \mathcal{A}(R)\bigr).
    \end{equation*}
    We can see that this is equivalent to
    \begin{equation}\label{Formula:Mvaluedftn_eq00}
        \forall a\in A \exists s\in S [\exists b\in B (s=\langle a,b\rangle \land \langle a,b\rangle\in R)].
    \end{equation}
    by the definition of $\mathcal{A}$.
    Hence the above statement is equivalent to $\forall a\in A\exists b\in B[ \langle a,b\rangle\in R\cap S]$, which means $R\cap S\colon A\rightrightarrows B$.
    The proof of the second statement is similar: $\mathcal{A}(R)\colon A\leftleftarrows S$ is equivalent to
    \begin{equation*}
        \forall s\in S \exists a\in A [\langle a,s\rangle \in \mathcal{A}(R)].
    \end{equation*}
    and by unpacking $\mathcal{A}$, we can see that the previous formula is equivalent to
    \begin{equation*}
        \forall s\in S \exists a\in A \bigl[\exists b\in B \bigl(s=\langle a,b\rangle \land \langle a,s\rangle\in \mathcal{A}(R)\bigr)\bigr],
    \end{equation*}
    which is readily equivalent to
    \begin{equation*}
        \forall s\in S \exists a\in A\exists b\in B \bigl[ s =\langle a,b\rangle \in R\bigr],
    \end{equation*}
    which is equivalent to $S\subseteq R$ since $R\subseteq A\times B$.
\end{proof}

\begin{proof}[Proof of \autoref{Proposition:SubsetCollectionEquivalentFormulations}]
    For the forward direction, let us apply the Subset Collection to $R_u=u$. 
    For the other direction, let $R_u\colon a\rrarrows b$ be a class multi-valued function with a parameter $u$. Here we can assume without loss of generality that $R_u\subseteq a\times b$.
    
    Then $\mathcal{A}^{a,b}(R_u)\colon a\rrarrows a\times b$. Now apply Strong Collection to $\mathcal{A}^{a,b}(R_u)$ to get a subimage $r$ of $\mathcal{A}^{a,b}(R_u)$. Hence by \autoref{Lemma:PrelimAdjuectmentFtn}, we have that $r\subseteq R_u$ and $r\colon a\rrarrows b$.
    Thus, if $c$ is full in subimages of multi-valued functions from $a$ to $b$, and $d\in c$ is a subimage of $r$, then we can see that $d$ is also a subimage of $R_u$. Thus $c$ witnesses Subset Collection for $R_u$.
\end{proof}

There is another form of Subset Collection known as \emph{Fullness}:

\begin{definition}
	\emph{Fullness} is the following statement:
	\begin{equation}\label{Axiom:Fullness}
		\forall A,B\exists \mathcal{C}\forall R (R\colon A\rightrightarrows B)\implies \exists S\in \mathcal{C} (S\subseteq R\land S\colon A\leftleftarrows B).
	\end{equation}
	In other words, Fullness states that for given sets $A$ and $B$, we can find a collection $\mathcal{C}$ of multi-valued functions such that every multi-valued function $R\colon A\rrarrows B$ is weakly uniformized by an element of $\mathcal{C}$.
	We call $\mathcal{C}\subseteq \mv(A,B)$ is \emph{full in $\mv(A,B)$} if $\mathcal{C}$ satisfies the condition in \eqref{Axiom:Fullness}.
\end{definition}

It is well-known that Subset Collection implies Fullness. Moreover, $\mathsf{CZF}^-$ can prove that Subset Collection and Fullness are equivalent. We introduce its proof for later reference. The following lemma is useful to prove the equivalence:

\begin{proposition}[$\mathsf{CZF}^-$] \phantom{a}
\begin{enumerate}
    \item Subset Collection implies Fullness.
    \item Fullness implies Subset Collection.
\end{enumerate}
\end{proposition}
\begin{proof}\pushQED{\qed}
\begin{enumerate}
    \item Let $A$ and $B$ be sets. 
    Apply Subset collection to $A$, $A\times B$ and $R_u= \mathcal{A}(u)$.
    Then we can find $\mathcal{C}$ such that
    \begin{equation*}
        \forall R (\mathcal{A}(R)\colon A\rightrightarrows A\times B)\implies
        \exists S\in \mathcal{C} (\mathcal{A}(R)\colon A\lrlrarrows S).
    \end{equation*}
    \autoref{Lemma:PrelimAdjuectmentFtn} implies $\mathcal{A}(R)\colon A\rightrightarrows A\times B$ is equivalent to $R\colon A\rightrightarrows B$, 
    and $\mathcal{A}(R)\colon A\lrlrarrows S$ is equivalent to
    $S\subseteq R \land R\cap S\colon A\rightrightarrows B$.
    By \autoref{Lemma:PrelimAdjuectmentFtn}, we have
    \begin{equation*}
	    \forall R (R\colon A\rightrightarrows B) \implies \exists S\in \mathcal{C} (S\subseteq R \land S\colon A\rightrightarrows B),
    \end{equation*}
    Therefore, $\mathcal{C}$ witnesses Fullness.
    
    \item Take $\mathcal{C}$ which is full in $\mv(A,B)$.
    Let $\langle R_u \mid u\in V\rangle$ be a class family and $R_u\colon A\rightrightarrows B$.
    Then $\mathcal{A}(R_u)\colon A\rightrightarrows A\times B$.
    By Strong Collection (see (2.2.15) of \cite{ZieglerPhD}), there is $S$ such that $A(\mathcal{R}_u)\colon A\lrlrarrows S$, which is equivalent to $S\subseteq R_u$ and $S\colon A\rightrightarrows B$.
    
    Since $\mathcal{C}$ is full in $\mv(A,B)$, there is $S'\in\mathcal{C}$ such that $S'\subseteq S$ and $S'\colon A\rightrightarrows B$. 
    From $S'\subseteq R_u$ and $S'\colon A\rightrightarrows B$, we can conclude $R_u\colon A\rightrightarrows \ran S'$ and $R_u\colon \ran S'\rightrightarrows A$. Therefore, the set $\{\ran S \mid S\in\mathcal{C}\}$ witnesses Subset Collection.
    \qedhere
\end{enumerate}
\end{proof}

\subsection{Variations of Powersets}
In this subsection, we discuss some special subclasses of the powerclass of a given set. 
We will see later that $\mathsf{DCom}$ for certain types of sets implies the existence of these subclasses.
Conversely, we will also assert that the existence of a certain subclass of a powerclass implies $\mathsf{DCom}$ for some sets.
\begin{definition} Let $a$ be a set.
\begin{itemize}
    \item The \emph{powerclass of $a$}, $\mathcal{P}(a)$ is the collection of
    all subsets of $a$. That is, $\mathcal{P}(a) = \{b\mid b\subseteq a\}.$
    \item The \emph{negative powerclass of $a$}, $\mathcal{P}^{\lnot\lnot}(a)$
    is the class $\{b\mid \forall x\in a[ \lnot\lnot(x\in a)\to x\in a]\}$.
    \item The \emph{class of decidable subsets} $\Dec(a)$ is the class
    $\Dec(a) = \{b\mid \forall x\in a[ x\in b\lor \lnot (x\in b)]\}$.
    \item The \emph{class of weakly decidable subsets} $\WDec(a)$ is the class
    $\WDec(a) = \{b\mid \forall x\in a[ \lnot(x\in b)\lor \lnot\lnot(x\in b)]\}$.
\end{itemize}
If each class is a set, we call each of them \emph{power set}, \emph{negative power set}, and \emph{set of (weakly) decidable sets}, respectively.
\end{definition}
Negative powerclass was introduced by Gambino \cite{GambinoThesis} to describe the negative variation of $\mathsf{IZF}$ named $\mathsf{IZF}^{\lnot\lnot}$. 
We reduce the existence of a variation of power sets to the existence of a variation of the power set of 1 under the presence of exponential sets ${}^ab = \{f\subseteq a\times b\mid f\colon a\to b\}$. 
In particular, we prove that $\Dec(a)$ is always a set since $\Dec(1)=2$ is a set:
\begin{proposition}[$\mathsf{CZF^- + Exp}$]\label{Proposition:VariationPower}
    Let $F$ be the one of $\mathcal{P}$, $\mathcal{P}^{\lnot\lnot}$, $\Dec$ and $\WDec$.
    If $F(1)$ exists, then $F(a)$ exists for every $a$.
\end{proposition}
\begin{proof}
    The main idea of the proof is the same regardless of $F$. We can see that if $F(a)$ is a $\Delta_0$-subclass of a set, then we can define $F(a)$ by using $\Delta_0$-Separation.
    Therefore, it suffices to prove 
    \begin{equation*}
    	F(a)\subseteq \{f^{-1}\{1\}\mid f:a\to F(1)\}
    \end{equation*}
    because the right-hand side is a set by Replacement and Exponentiation.
    
    For each $b\in F(a)$, consider its characteristic function $\chi_b:a\to\mathcal{P}(1)$ defined by $\chi_b(x) = \{0\mid x\in b\}$.
    We can see that $b=\chi_b^{-1}\{1\}$. 
    It remains to show that $\ran \chi_b \subseteq F(1)$, which will ensure $\chi_b$ is a function of codomain $F(1)$. Its proof depends on $F$, but all proofs employ the following equivalence:
    
    \begin{equation}\label{Formula:CharFtnEquivalence}
    	x\in b\iff 0\in \chi_b(x).
    \end{equation}
    Let us examine its proof case by case.
    \begin{enumerate}
        \item Case $F=\mathcal{P}$: Obvious.
        
        \item Case $F=\mathcal{P}^{\lnot\lnot}$: Let $b\in \mathcal{P}^{\lnot\lnot}(a)$. Then we have $\lnot\lnot(x\in b)\to x\in b$ for all $x\in a$, and this is equivalent to
        \begin{equation*}
            \lnot\lnot(0\in \chi_b(x))\to 0\in \chi_b(x)
        \end{equation*}
        by \eqref{Formula:CharFtnEquivalence}. Hence $\chi_b(x)\in \mathcal{P}^{\lnot\lnot}(1)$.

        \item Case $F=\WDec$: Similarly, we will show $\ran \chi_b \subseteq \WDec(1)$ if $b\in \WDec(a)$. By the assumption on $b$, we have
        \begin{equation*}
        	\lnot(x\in b) \lor \lnot\lnot(x\in b),
        \end{equation*}
        for all $x\in a$, which turns out to be equivalent to
        \begin{equation*}
        	\lnot(0\in \chi_b(x)) \lor \lnot\lnot(0\in \chi_b(x))
        \end{equation*}
        by \eqref{Formula:CharFtnEquivalence}. Therefore, $\chi_b(x) \in \WDec(1)$.

        \item Case $F=\Dec$: Similar to the case $\WDec$, so we omit it. \qedhere 
    \end{enumerate}
\end{proof}

Note that \autoref{Proposition:VariationPower} still holds even if we replace $F(1)$ with $F(b)$ for any singleton $b$. Moreover, if $a\subseteq b$ and $F(b)$ exists, then $F(a)$ also exists by $\Delta_0$-Separation. Hence the assertion `$F(a)$ exists for all $a$' is equivalent to `$F(a)$ exists for some \emph{inhabited} $a$'.

We will describe $\WDec(1)$ in detail for future use. Suppose that $x\in \WDec(1)$.
Then either $0\notin x$ or $\lnot\lnot(0\in x)$. In the former case, $x$ must be
empty since 0 is the only possible element of $x$. 
If  $\lnot\lnot(0\in x)$, we have $\lnot\lnot(x=1)$ since $\lnot\lnot(x\subseteq 1)$ holds. 
Conversely, it is easy to check that the empty set $0$ and a set $x\subseteq 1$ satisfying $\lnot\lnot(x=1)$ are members of $\WDec(1)$. Therefore, we have the following equation:
\begin{lemma}\label{Formula:WDec1}\pushQED{\qed}
	\begin{equation}
    	\WDec(1) = \{x\subseteq 1 \mid x=0\lor \lnot\lnot(x=1)\}. \qedhere
	\end{equation}
\end{lemma}

\subsection{Fully inductive relations}
Well-founded relations are useful in $\mathsf{ZF}$ or its extensions because they allow recursive definition. However, this is not true for weaker theories that do not admit Full separation. For example, we do not know whether $\mathsf{CZF}$ proves the transitive collapsing function
    \begin{equation*}
    	\pi(x) := \{\pi(y) \mid y\prec x\}
    \end{equation*}
on a well-founded set $\langle A,\prec\rangle$ is a set function.
Hence $\mathsf{CZF}$ cannot define subsets of $A$ like 
\begin{equation}\label{Formula:ExampleSet000}
	\{x\in A\mid \exists f\in {^x}\pi(x) (\text{$f$ is bijective})\},
\end{equation}
due to the lack of Full separation. It also means we cannot apply the $\prec$-induction scheme to the defining formula of \eqref{Formula:ExampleSet000}.
It motivates us to define a stronger notion of a well-founded set:
\begin{definition}
    Let $A$ be a class and $\prec$ be a binary relation over $A$.
    We call $\prec$ is \emph{fully inductive} if it is \emph{uniformly extensional} in the sense that $\ext_\prec(x) := \{y\in A : y\prec x\}$ is a set which is definable uniform to $x$ and satisfies the following $\prec$-induction scheme for formulas $\phi$:
        \begin{equation*} 
			[(\forall y\in A ( y\prec x \to \phi(y))\to \phi(x))]\to \forall x\in A \phi(x) 
        \end{equation*}
    We call $\prec$ is \emph{$\Delta_0$-definable fully inductive} if $\prec$ is moreover $\Delta_0$-definable.
\end{definition}
We can show that $\prec$ is $\Delta_0$-definable fully inductive if it satisfies $\prec$-induction scheme and $\ext_\prec(x)$ exists for all $x$. Thus we can drop the uniform definability of $\ext_\prec$ for $\Delta_0$-definable fully inductive relations.
\begin{lemma}
    A relation $\prec$ over a class $A$ is $\Delta_0$-definable fully inductive if and only if $\prec$ is $\Delta_0$-definable, $\ext_\prec(x)$ exists for all $x$, and the $\prec$-induction scheme is satsfied.
\end{lemma}
\begin{proof}
   Let $\prec$ be a $\Delta_0$-definable relation which is extensional and satisfies the $\prec$-induction scheme.
   Then the following formula describes $z=\ext_\prec(x)$:
   \begin{equation*}
       \exists a [\forall y\in A (y\prec x \to y\in a)] \land [\forall y\in A (y\in z\leftrightarrow y\in a\land y\prec x)]. \qedhere 
   \end{equation*}
\end{proof}

We concentrate on $\Delta_0$-definable fully inductive relations in this article since every fully inductive relation that will appear in this paper is $\Delta_0$-definable.

$\in$-induction states that the membership relation $\in$ is fully inductive. Fully inductive relations are nicely behaved even when our background theory is weak, unlike well-founded relations.
For example, we can prove that fully inductive relations satisfy well-founded recursion for arbitrary class functions:
\begin{theorem}[Well-founded recursion, $\mathsf{CZF}^-$]
    If $\prec$ is a fully inductive relation over a class $A$ and
    $G\colon A\to A$ be a class function, then there is a class function $F\colon A\to V$
    such that $F(x)=G(F\upharpoonright\operatorname{ext}_\prec(x))$ for all $x\in A$.
\end{theorem}
Before describing the proof, the readers are reminded that $F\upharpoonright u=\{\langle y,F(y)\rangle :y\in u\}$, which exists by Replacement.
\begin{proof}
Let $\Phi(f)$ be the following formula:
    \begin{equation*}
        [\text{$f$ is a function}]\land 
        \bigl[\forall x\in\dom f \bigl( \ext_\prec(x)\subseteq \dom f\bigr)\bigr]
        \land \bigl[\forall z\in\dom f \bigl( f(z) = G(f\restricts \ext_\prec(z) )\bigr)\bigr]
    \end{equation*}
    We will prove the following two lemmas before constructing the
    desired $F$.
    \begin{lemma}\label{Lemma:RecursionYamada} 
        If $\Phi(f)$ and $\Phi(g)$ then
        \begin{equation*}
        	\forall x\in A \bigl( x\in\dom f\cap\dom g\to f(x)=g(x)\bigr).\qedhere
        \end{equation*}
    \end{lemma} 
	\begin{proof}
		We will use the $\prec$-induction on $x$, which is possible as $\prec$ is
		a fully inductive relation on $A$.
		Assume that our theorem holds for all $y\prec x$. Now assume that $x\in \dom f\cap \dom g$.
		By the inductive assumption, we have $f\restricts \ext_\prec(x) = g\restricts \ext_\prec(x)$.
		Since $\Phi(f)$ and $\Phi(g)$ holds, we have $f(x)=G(f\restricts \ext_\prec(x))=G(g\restricts \ext_\prec(x)) = g(x)$.
	\end{proof}
    
    \begin{lemma}
        For each $x\in A$ there is a function $f$ such that $\Phi(f)$ and
        $x\in\dom f$.
    \end{lemma}
    \begin{proof}
		We will prove it by $\prec$-induction on $x$. Suppose that for each 
		$y\prec x$ there is a function $f$ such that $\Phi(f)$. 
		By Strong Collection, we can find a set $A$ such that
		\begin{equation*}
			\bigl[\forall y\prec x\exists f\in A \bigl( \Phi(f) \land y\in \dom f\bigr)\bigr]
            \land \bigl[\forall f\in A \exists y\prec x \bigl( \Phi(f)\land y\in\dom f\bigr)\bigr].
		\end{equation*}
		Especially, every $f\in A$ satisfies $\Phi(f)$.
		Let $f_0 := \bigcup A$. We will check that $\Phi(f_0)$ holds.
		By \autoref{Lemma:RecursionYamada}, 
		$f_0$ is a function. Moreover, it is easy to check the remaining 
		conditions of $\Phi(f_0)$. Therefore $\Phi(f_0)$ holds.

		Now take 
		\begin{equation*}
			f = f_0\cup \bigl\{\bigl\langle x, G(f_0\restricts \ext_\prec(x) )\bigr\rangle\bigr\}.
		\end{equation*}
		The only non-trivial part to prove $\Phi(f)$ is to verify 
		$f$ is a function, especially the well-definedness of $f(x)$. 
		If $\langle x,y\rangle,\langle x,z\rangle\in f$, then 
		either both of them are in $f_0$, or one of them is in 
		$\bigl\{\bigl\langle x, G(f_0\restricts \ext_\prec(x) )\bigr\rangle\bigr\}$.
		In the latter case, we can see
		$y = z = G (f_0\upharpoonright \ext_\prec(x))$. Hence $f$ is a function.
    \end{proof}
    
    The main theorem follows by letting $F=\bigcup\{f : \Phi(f)\}$.
\end{proof}

Is there an example of a fully inductive relation? $\in$-induction states $\in$ over $V$ is fully inductive. It is also known that $\in$ over
the class of ordinals $\Ord$ is fully inductive. (See \cite[Lemma 9.4.3.]{AczelRathjen2010}.) 
We will see further examples later.

\subsection{Inductive definitions}
Some definitions in mathematics are \emph{impredicative} in the sense that the definition of some object refers to itself, either directly or indirectly. Some impredicative definitions take the following form: 

An impredicative description justifies the following definition for a property $P$: an object $H$ is the smallest object that satisfies $P$. 
Predicative constructive systems like $\mathsf{CZF}$ do not allow impredicative definitions. Fortunately, we can define some types of `smallest objects' over $\mathsf{CZF}$

Let $C$ be a class and $\Gamma$ be a monotone operation.
A collection of classes $\langle C^\alpha \mid \alpha\in\Ord\rangle$ is a \emph{$\Gamma$-hierarchy for $C$} if it satisfies $C = \bigcup_{\alpha\in\Ord}$ and $C^\alpha = \Gamma(C^{\in\alpha})$, where $C^{\in\alpha} = \bigcup_{\beta\in\alpha} C^\beta$.

The proof of the theorem is available in \cite{AczelRathjen2010}, so we omit it:
\begin{theorem}[Class Inductive Definition Theorem, $\mathsf{CZF}^-$]
    \label{Theorem:ClassInductiveDefThm}
    Let $\Phi$ be an inductive definition. Then we can find a smallest $\Phi$-closed class $I_{\Phi}$ and a $\Gamma_\Phi$-hierarchy $\langle I_\Phi^\alpha \mid\alpha\in\Ord\rangle$ for $I_\Phi$.
    Furthermore, each $I^\alpha_\Phi$ is a set if $\Gamma_\Phi(X)$ is a set for every set $X$.
\end{theorem}

\section{Introducing the Axiom of Double Complement}
\label{Section:DCom}

In this section, we will examine basic properties of the double complement operation. We will define the Axiom of Double Complement $\mathsf{DCom}$ and its opposites, named $\mathsf{NDCom}$ and $\mathsf{ADCom}$ that will be analyzed through the rest of the paper.

\begin{definition}
\begin{itemize}
    \item $\mathsf{DCom}$, the \emph{Axiom of Double Complement} states every set has a double complement:
    \begin{equation*}
        \forall x\exists y\forall z \bigl( \lnot\lnot(z\in x)\to z\in y\bigr).
    \end{equation*}
    
    \item $\mathsf{NDCom}$, the \emph{Axiom of Non-Double Complement} states there is a set whose double complement does not exist:
    \begin{equation*}
        \exists x\forall y \Bigl[ \lnot\bigl[\forall z\bigl( \lnot\lnot(z\in x)\to z\in y\bigr)
        \bigr]\Bigr]
    \end{equation*}
    
    \item $\mathsf{ADCom}$, the \emph{Axiom of Anti-Double-Complement} states every set which has a double complement is a subset of 1:
    \begin{equation*}
        \forall x  \Bigl[\bigl[ \exists y\forall z \bigl(\lnot\lnot z\in x\to z\in y\bigr)\bigr]\to x\subseteq 1\Bigr].
    \end{equation*}
\end{itemize}
\end{definition}

$\mathsf{ADCom}$ is motivated by \autoref{Proposition: 1has-a-DCom}, which claims $\mathsf{CZF}^-$ proves subsets of $1$ have a double complement. $\mathsf{ADCom}$ claims subsets of 1 are the only sets that have a double complement. Thus, we may view $\mathsf{ADCom}$ as a `maximal' form of the opposition of $\mathsf{DCom}$.

We will establish the existence of a double complement for some sets and find their representation as much as possible. It turns out that we can easily find and describe the double complement of subsets of 1. However, the situation is convoluted for subsets of more complicated sets.
We can prove the existence of the double complement for elements of $V_\omega$ under Powerset, and the existence of $V_\omega^\dcom$ under a semi-classical background.
Describing a double complement of a given set is more challenging, but shows an intriguing aspect: it turns out that the double complement of $2$ is the powerset of 1. 
We also discuss the consistency and compatibility of $\mathsf{DCom}$ and their opposites with constructive set theories.

\subsection{Finding and describing a double complement}
In this subsection, we will establish the existence of a double complement for some simple sets. Also, we discuss how to represent a double complement of a given set if possible.

Many natural examples that have a double complement do not only have a double complement, but also they coincide with their double complement. Thus we introduce a terminology for these types of sets:
\begin{definition}
    A set $x$ is \emph{stable} if $x=x^\dcom$.
\end{definition}

We will use the following basic properties of double complement frequently:
\begin{proposition}[$\mathsf{CZF}^-$]\label{Proposition:BasicsofDCom}
\begin{enumerate}
    \item If $x\subseteq y$ then $x^\dcom \subseteq y^\dcom$.
    \item $x\subseteq x^\dcom$.
    \item $\mathcal{P}(x)^\dcom \subseteq \mathcal{P}(x^\dcom)$. In particular, if $x$ is stable, then so is $\mathcal{P}(x)$.
\end{enumerate}
\end{proposition}
\begin{proof}
We only give a proof for the last statement. Let $y\in \mathcal{P}(x)^\dcom$, that is,
\begin{equation*}
    \lnot\lnot(\forall z: z\in y\to z\in x).
\end{equation*}
Since we have $\lnot\lnot\forall z\phi(z)\to \forall z[\lnot\lnot\phi(z)]$ and $\lnot\lnot(p\to q)\to (\lnot\lnot p\to\lnot\lnot q)]$, we can derive $\forall z\in y: \lnot\lnot(z\in x)$. That is, we have $y\subseteq x^\dcom$.
\end{proof}

It is natural to ask if there is a set that has a double complement.
Moreover, it would be better if we could find a concrete representation of the double complement of a given set. Unfortunately, evaluating the double complement for arbitrary sets is usually hard, unless we have an additional axiom like $\mathsf{\Delta_0\mhyphen LEM}$. 
Despite that, we can find a concrete representation of the double complement for some simple sets.
For example, it is easy to see that $0^\dcom = 0$. In general, we can evaluate the double complement of subsets of 1:
\begin{proposition}[$\mathsf{CZF}^-$]\label{Proposition: 1has-a-DCom}  
    If $x\subseteq 1$ then $x^\dcom = \{0\mid \lnot\lnot(0\in x)\}$.
\end{proposition}
\begin{proof}
    If $y\in x^\dcom$, then $\lnot\lnot(y=0 \land 0\in x)$ holds, 
    which is equivalent to
    \begin{equation*}
        \lnot\lnot(y=0) \land \lnot\lnot(0\in x).
    \end{equation*}
    Since $y=0$ is a negation of the statement `$y$ is inhabited', 
    $\lnot\lnot(y=0)$ is just $(y=0)$. 
    Therefore $x^\dcom \subseteq \{0\mid \lnot\lnot(0\in x)\}$.
    The reverse inclusion can be shown easily.
\end{proof}
In particular, we can see that $1^\dcom = 1$. 
Unfortunately, subsets of 1 are the only examples in which we can find their double complement over $\mathsf{CZF}^-$ without additional assumptions, since $\mathsf{CZF + ADCom}$ is consistent. (See \autoref{Section:ADCom} for its proof.)
However, we can find more sets that have a double complement if we allow Powerset.
\begin{proposition}\label{Proposition:StabilityOfVn} \phantom{a}
\begin{enumerate}
    \item \emph{($\mathsf{CZF^-}$)} For each $n\in\omega$, $V_n$ is stable. That is $V_n^\dcom=V_n$.
    \item \emph{($\mathsf{CZF+Pow}$)} If $x\in V_\omega$ then $x^\dcom$ exists.
\end{enumerate}
\end{proposition}
\begin{proof}
    Let us remind that von Neumann hierarchies are defined by $V_\alpha = \bigcup_{\beta<\alpha}\mathcal{P}(V_\beta)$. We claim $V_{n+1}= \mathcal{P}(V_n)$ for a natural number $n$ by induction on $n$. Then the desired theorem follows from \autoref{Proposition:BasicsofDCom}.
    
    This is obvious if $n=0$. If $V_{k+1} = \mathcal{P}(V_k)$ for all $k<n$, then
    \begin{equation*}
    \begin{array}{lll}
        V_{n+1} &= \bigcup_{k<n}\mathcal{P}(V_k) \cup \mathcal{P}(V_n) & \\
        &= \bigcup_{k<n}V_{k+1}\cup \mathcal{P}(V_n) 
        & (\because \text{ Inductive assumptions.})\\
        &= V_n \cup \mathcal{P}(V_n)
        & (\because \text{ $\langle V_k\mid k\in n\rangle$ 
        is increasing under $\subseteq$.})\\
        &= \mathcal{P}(V_n)
        & (\because \text{ $V_n$ is transitive so $V_n\subseteq\mathcal{P}(V_n)$.})
    \end{array}
    \end{equation*}
    Now the stability of $V_n$ easily follows from \autoref{Proposition:BasicsofDCom} and  induction on $n$.
    
    For the second assertion, if $x\in V_\omega=\bigcup_{n\in\omega} \mathcal{P}(V_n)$, then $x$ is a subset of some $V_n$. By Powerset, each $V_n$ is a set. In addition, we have $x^\dcom\subseteq V_n^\dcom=V_n$, and thus $x^\dcom$ exists by Bounded Separation.
    \qedhere
\end{proof}
As a corollary, we have
\begin{corollary}[$\mathsf{CZF}^-$]
    $\mathsf{ADCom}$ is incompatible with Powerset. \qedhere
\end{corollary}
\begin{proof}
	$V_2$ is a double complement of itself, but not a subset of $1$.
\end{proof}

We may ask if we can find a concrete representation of a double complement for simple sets other than stable sets, for example, for natural numbers and $\omega$. The bad news is that expressing a double complement of 3 under simple terms is unlikely to be possible. Despite that, we can find a precise description of the double complement for 2 and $\{1\}$. As promised, their double complements are connected with power sets:
\begin{theorem}[$\mathsf{CZF}^-$]\label{Theorem:DComof2}
    The double complement of 2 is $\mathcal{P}(1)$. Moreover, the double complement
    of $\{1\}$ is $\{x\subseteq 1\mid \lnot\lnot(x=1)\}$.
\end{theorem}
\begin{proof}
	 We only give a proof for the former statement, as the proof of the latter statement is analogous.
	Since $2\subseteq\mathcal{P}(1)$, we have $2\subseteq \mathcal{P}(1)^\dcom = \mathcal{P}(1)$. It remains to show the reverse inclusion. 
	Let $x\subseteq 1$. We must show that
		\begin{equation*}
			\lnot\lnot(x=0 \lor x=1).
		\end{equation*}
	But this follows from $\lnot\lnot(0\in x\lor 0\notin x)$ and $\lnot\lnot(x\subseteq 1)$.
\end{proof}

That is, $\mathsf{DCom}$ shows $2^\dcom=\mathcal{P}(1)$ exists. Combining with Subset Collection, we have the following corollary:
\begin{corollary}[$\mathsf{CZF}$]  \label{Corollary:DComToPow}
    $\mathsf{DCom}$ implies $\mathsf{Pow}$.
\end{corollary}
\begin{proof}
	\autoref{Theorem:DComof2} implies $\mathcal{P}(1)$ exists.
	Moreover, \autoref{Proposition:VariationPower} states that the existence of $\mathcal{P}(1)$ implies $\mathsf{Pow}$.
\end{proof}
It is natural to ask whether Powerset proves the Axiom of Double Complement. It will turn out that $\mathsf{IZF+NDCom}$ is consistent, so the answer is negative.

Let us discuss the possible size of the double complement of sets before concluding this subsection. A double complement of a given set may not exist and could be a proper class. We know that classical set theories prove all proper classes are quite large in the sense that they are \emph{inexhaustible}.
\begin{definition}
    A class $A$ is \emph{inexhaustible} if we can find $y\in A$ such that $y\notin x$ for any set $x$. 
\end{definition}

We can show that $\mathsf{CZF}^-$ proves the class of all sets $V$, and the class of all ordinals $\Ord$ is inexhaustible: See \cite[Example 6.24]{ZieglerPhD} for its proof.
We may ask whether the double complement of a set can be inexhaustible, and the answer is negative:

\begin{proposition}[$\mathsf{CZF^-}$]\label{Proposition:NonExhausability}
    If $A$ is a set, then $A^\dcom$ is not inexhaustible.
\end{proposition}
\begin{proof}
    No $y$ can satisfy both $y\in x^\dcom$ and $y\notin x$, which have to exist if $x^\dcom$ is inexhaustible.
\end{proof}
Note that it resolves an open question stated in Ziegler \cite[Open Question 6.25]{ZieglerPhD}, which asks the consistency of existence of two sets $A$ and $B$ such that $\mathcal{P}(A)$ is inexhaustible while $\mathcal{P}(B)$ is not:
\begin{corollary}
    Working over $\mathsf{CZF}$ with the Axiom of Subcountability, which states every set is an image of a subset of $\omega$, 
    $\mathcal{P}(\omega)$ is inexhaustible but $\mathcal{P}(1)$ is not.
\end{corollary}
\begin{proof}
	The Axiom of Subcountability proves $\mathcal{P}(\omega)$ is inexhaustible (See \cite[Example 6.24]{ZieglerPhD}.) However, \autoref{Proposition:NonExhausability} proves $\mathcal{P}(1)=2^\dcom$ is never inexhaustible.
\end{proof}

\subsection{Consistency and compatibility of $\mathsf{DCom}$ over $\mathsf{CZF}$}
\autoref{Corollary:DComToPow} shows that $\mathsf{DCom}$ implies $\mathsf{Pow}$, so the consistency of $\mathsf{CZF+DCom}$ implies that of $\mathsf{CZF+Pow}$. The following result shows the precise consistency strength of $\mathsf{CZF+DCom}$ and $\mathsf{CZF+NDCom}$:

\begin{proposition} \label{Proposition:DComOverCZF}
    \begin{enumerate} 
        \item $\mathsf{CZF+DCom}$ is equiconsistent with $\mathsf{CZF+Pow}$.
        \item $\mathsf{CZF}$ is equiconsistent with $\mathsf{CZF + NDCom}$.
    \end{enumerate}
\end{proposition}
\begin{proof}
\begin{enumerate}
	\item One direction directly follows from \autoref{Corollary:DComToPow}.
    The other direction follows from the equiconsistency of $\mathsf{CZF+\Delta_0\mhyphen LEM}$ and $\mathsf{CZF+Pow}$ in \cite{Rathjen2012Power} since $\mathsf{\Delta_0\mhyphen LEM}$ proves every set is stable.
    \item In fact, the Axiom of Subcountability proves $\lnot\mathsf{Pow}$. In $\mathsf{CZF}$, $\mathsf{Pow}$ holds if and only if $\mathcal{P}(1)$ exists. 
    Since $2^\dcom=\mathcal{P}(1)$, 2 is an instance of $\mathsf{NDCom}$.
    Moreover, \cite{Rathjen2002Choice} proves $\mathsf{CZF}$ and $\mathsf{CZF}$ + Axiom of Subcountability are equiconsistent.
    \qedhere
\end{enumerate}
\end{proof}
What if we weaken the background theory from $\mathsf{CZF}$ to $\mathsf{CZF^-}$?
We know that $\mathsf{CZF}$ proves $\mathsf{DCom}$ implies $\mathsf{Pow}$ by \autoref{Proposition:VariationPower}, but Subset Collection has a critical role for its proof.
Indeed, adding $\mathsf{DCom}$ into $\mathsf{CZF}^-$ does not increase the consistency strength.
\begin{proposition}
    $\mathsf{CZF^-}$, $\mathsf{CZF^-+DCom}$, $\mathsf{CZF^-+NDCom}$ are all equiconsistent.
\end{proposition}
\begin{proof}
    By \cite[Lemma 2.4 and Theorem 4.2]{Rathjen2012EP}, $\mathsf{CZF}^-$ and $\mathsf{CZF}$ are equiconsistent. Therefore $\mathsf{CZF^-}$ + $\mathsf{NDCom}$ and $\mathsf{CZF^-}$ are equiconsistent.
    Moreover, \cite[Theorem 4.2]{Gambino2006} proves $\mathsf{CZF}^-$ and $\mathsf{CZF}^-$ + $\mathsf{\Delta_0\mhyphen LEM}$ are equiconsistent.
    From the fact that $\mathsf{\Delta_0\mhyphen LEM}$ proves every set is stable, we have the equiconsistency between $\mathsf{CZF}^-$ and $\mathsf{CZF^-+DCom}$.
\end{proof}

We have not established the consistency strength of $\mathsf{ADCom}$.
We will show in \autoref{Section:ADCom} that $\mathsf{CZF+ADCom}$ and $\mathsf{CZF}$ are equiconsistent.
The proof of \autoref{Proposition:DComOverCZF} would establish the consistency strength of $\mathsf{CZF+ADCom}$ if the Axiom of Subcountability implies $\mathsf{ADCom}$ over $\mathsf{CZF}$.
We do not know whether the Axiom of Subcountability proves $\mathsf{ADCom}$ or not:
\begin{question}\label{Question:Subcountability}
    Does the combination of $\mathsf{CZF}$ and the Axiom of Subcountability prove $\mathsf{ADCom}$?
\end{question}

\autoref{Proposition:DComOverCZF} shows that in $\mathsf{CZF}$, the existence of $2^\dcom$ implies Powerset.
What happens if we assume the existence of the double complement of another set? The next proposition says that the existence of $a^\dcom$ for certain $a$ implies the existence of variants of power sets:
\begin{proposition}[$\mathsf{CZF}^-$]\label{Proposition:SomeSetsWithoutDCom}
     Let $a$ be a set, and suppose that $a^\dcom$ exists.
    \begin{enumerate}
        \item If $a$ has two \emph{apart} elements, that is, 
        if there are $x,y\in A$ such that
        \begin{equation*}
        	\exists z \bigl[ (z\in x \land z\notin x)\lor (z\notin x\land z\in y)\bigr]
        \end{equation*}
        then $\mathcal{P}(1)$ exists.
        \item If $a$ has two \emph{nonequal} elements, that is, if $a$ contains $x$ and $y$
        such that $x\neq y$, then $\mathcal{P}^{\lnot\lnot}(1)$ exists.
        \item If $a$ has an inhabited element, then $\WDec(1)$ exists.
    \end{enumerate}
\end{proposition}
\begin{proof}
    \begin{enumerate}
        \item Without loss of generality, assume that there is $t$ such that
        $t\in x$ but $t\notin y$. Define $w_c$ by
            \begin{equation*}
            	w_c = \{z\in x \mid 0\in c\}\cup \{z\in y \mid 0\notin c\}.
            \end{equation*}
        Then $\lnot\lnot(w_c\in a)$. Now we will prove that
        $c=\{0\mid w_c=x\}$ for all $c\subseteq 1$: one inclusion is obvious.
        For the reverse inclusion, let $w_c=x$. Then $t\in w_c$, which is
        equivalent to
            \begin{equation*}
            	(t\in x\land 0\in c)\lor (t\in y\land 0\notin c).
            \end{equation*}
        We can exclude the latter case as $t\notin y$. Hence $0\in c$ and we have
        $\{0\mid w_c=x\}\subseteq c$. Therefore we have
        $\mathcal{P}(1)\subseteq \{\{0\mid z=x\}\mid z\in a^\dcom\}$.
        
        \item For each $c\subseteq 1$ define $w_c$ as before.
        We can deduce $\lnot\lnot (w_c\in a)$ from $\lnot\lnot(0\in c\lor 0\notin c)$.
        We claim that $\mathcal{P}^{\lnot\lnot}(1) \subseteq \{\{0\mid z=x\}\mid 
        z\in a^\dcom\}$. This would follow from $c = \{0\mid w_c=x\}$ for stable
        subset $c$ of 1. $c\subseteq \{0\mid w_c= z\}$ is obvious.
        For the reverse inclusion, suppose that $w_c=x$. Since $x\neq y$, 
        $w_c\neq y$ holds so we have $\lnot\lnot(0\in c)$.  
        Therefore $0\in c$ by stability of $c$.
        
        \item Suppose that $x\in a$ and $y\in x$. For each $c\in \WDec(1)$, 
        define $w_c = (x\setminus\{y\})\cup \{y\mid 0\in c\}$.
        We can show that $\WDec(1) \subseteq \{\{0\mid z=x\} : z\in a^\dcom\}$ by verifying $c=\{0\mid w_c=x\}$ for each $c\in \WDec(1)$. The proof is similar to previous cases, so we omit the details. \qedhere
    \end{enumerate}
\end{proof}

We may regard the first part of \autoref{Proposition:SomeSetsWithoutDCom} as a generalization of \autoref{Theorem:DComof2}. As a corollary of \autoref{Proposition:SomeSetsWithoutDCom}, we have
\begin{corollary}
    The following theories are equiconsistent.
    \begin{enumerate}
        \item \label{Corollary:OtherDcom1} $\mathsf{CZF}+$`There is a set $a$ such that $a$ has two apart elements, and $a^\dcom$ exists,'
        \item \label{Corollary:OtherDcom2} $\mathsf{CZF}+$`There is a set $a$ such that $a$ has two nonequal elements, and $a^\dcom$ exists,'
    	
        \item \label{Corollary:OtherDcom3} $\mathsf{CZF+Pow}$
        \item \label{Corollary:OtherDcom4} $\mathsf{CZF}+\mathcal{P}^{\lnot\lnot}(1)\text{ exists}$.
    \end{enumerate}
\end{corollary}
\begin{proof}
    \eqref{Corollary:OtherDcom1} obviously implies \eqref{Corollary:OtherDcom2}. \eqref{Corollary:OtherDcom2} implies \eqref{Corollary:OtherDcom4} by \eqref{Proposition:SomeSetsWithoutDCom}. By \cite[Theorem 1.1]{Rathjen2012Power}, \eqref{Corollary:OtherDcom4} and \eqref{Corollary:OtherDcom3} have the same consistency strength.
    Moreover, \eqref{Corollary:OtherDcom3} implies \eqref{Corollary:OtherDcom1} by \autoref{Theorem:DComof2}.
\end{proof}

It is unclear that the last part of \autoref{Proposition:SomeSetsWithoutDCom} implies any consistency strength results, since we do not know the consistency strength of $\mathsf{CZF}$ + `$\WDec(1)$ exists'.
However, it gives a hint to characterize $\mathsf{ADCom}$ in terms of weakly decidable sets, which will turn out to be useful:

\begin{theorem}($\mathsf{CZF}$)\label{Theorem:ADComAndWDec}
    $\mathsf{ADCom}$ is equivalent to the non-existence of $\WDec(1)$.
\end{theorem}
\begin{proof}
    By \autoref{Proposition:SomeSetsWithoutDCom}, the absence of $\WDec(1)$ implies the absence of the double complement of sets that contain an inhabited element.
    Hence if $a$ has a double complement, then every element of $a$ must be empty. That is, we have $a\subseteq \{0\} = 1$. This is the very definition of $\mathsf{ADCom}$.

    Conversely, assume that $\mathsf{ADCom}$ holds. Then $\{1\}$ does not have a double complement. Therefore, $\{x\subseteq 1 \mid \lnot\lnot(x=1)\}$ is not a set. Hence $\WDec(1) = \{0\}\cup \{x\subseteq 1 \mid \lnot\lnot(x=1)\}$ is also not a set.
\end{proof}

\section{Lubarsky's Kripke models}
\label{Section:KripkeModel}
In this section, we will develop basic facts on Kripke models over $\mathsf{CZF}$ and $\mathsf{IZF}$.
There are many possible definitions of Kripke models we can take: For example, \cite{Friedman1985, Iemhoff2010} employ their own Kripke model. We follow Lubarsky's Kripke models that appear in \cite{Lubarsky2005, HendtlassLubarsky2016}.
Our Kripke models may be a special case of Heyting-valued models that were introduced by \cite{Gambino2006, Bell2014}, but we prefer to develop the theory of Kripke models separately. 
In this section, we work over $\mathsf{CZF}^-$ unless specified.

\subsection{Kripke models over set frames}
We need a \emph{frame}, which is a partially ordered set with the least element $\bot$, to construct a Kripke model.
We will denote frames as $(\mathbb{P},\le)$, or just $\mathbb{P}$ if the order relation is clear. We use variables $p$, $q$, $r$, $s$, $\cdots$ for elements of $\mathbb{P}$. For each $p\in\mathbb{P}$, $\uparrow p$ denotes the upper set $\{q\in\mathbb{P} \mid q\ge p\}$ of $p$.
\begin{definition}[Kripke models]
\label{Definition:KripkeInformal}
    Let $V$ be a model of $\mathsf{CZF}$ (or $\mathsf{IZF}$) and $\mathbb{P}\in V$ be a frame. 
	The domain of the Kripke model $V^\mathbb{P} =\bigcup_{p\in\mathbb{P}} V^\mathbb{P}(p)$ on $\mathbb{P}$ with transition functions $\tau_{pq} : V^\mathbb{P}(p)\to V^\mathbb{P}(q)$ for $p\le q$ satisfy the following conditions: for ordinals $\beta\subseteq \alpha$ and $p\le q\le r$,
	\begin{enumerate}[label=(\alph*)]
    	\item \label{Labeling:SetKripke1}
    	$V^\mathbb{P}_\beta(p)\subseteq V^\mathbb{P}_\alpha(p)$ for all $p\in\mathbb{P}$.
    	\item \label{Labeling:SetKripke2}
    	For $x\in V^\mathbb{P}_\alpha(p)$, $\tau_{pq}(x)=x\restricts (\uparrow q)$, and 
   	 	$\tau_{pq}\restricts V_\alpha^\mathbb{P}(p)$ is a function from $V_\alpha^\mathbb{P}(p)$ to $V_\alpha^\mathbb{P}(q)$.
    	\item \label{Labeling:SetKripke3}
    	$\tau_{pr}=\tau_{qr}\circ\tau_{pq}$ and $\tau_{pp}$ is the identity function.
	    \item \label{Labeling:SetKripke4}
    	$V^\mathbb{P}_\alpha(p)$ is a class of all functions $x$ of domain $\uparrow p$ such that 
    	\begin{itemize}
        	\item $x(q)\subseteq \bigcup_{\beta\in\alpha} V^\mathbb{P}_\beta(q)$ for all $q\ge p$ and
        	\item $\tau_{qr}''[x(q)]\subseteq x(r)$ for all $p\le q\le r$.
    	\end{itemize}
    (We will call the last condition for $x$ the \emph{monotonicity} condition.)
	\end{enumerate}
\end{definition}

We call elements of $V^\mathbb{P}$ a \emph{Kripke sets over $\mathbb{P}$} or \emph{$\mathbb{P}$-names}. We view $\mathbb{P}$ as the set of \emph{nodes of worlds}, which describes the possible flow of the future of the set-theoretic world. 
We regard $V^\mathbb{P}(p)$ as a set-theoretic world at node $p$, and we want to assume that every set in the Kripke model grows as $p$ flows; that is, if $x\in V^\mathbb{P}(p)$ and $p\le q$, then $x(p)$ is `contained in' $x(q)$.
We will codify $x\in V^\mathbb{P}(p)$ as a function of domain $\uparrow p$, so that $x$ has a determined shape at node $q$ for $q\ge p$.
Furthermore, every member of $x(p)$ is also a member of $V^\mathbb{P}(p)$. However, such a definition makes the direct comparison between $x(p)$ and $x(q)$ for $p\le q$ by the inclusion relation impossible, because members of $x(p)$ and those of $x(q)$ are functions of different domains.
The transition functions $\tau_{pq}$ mend this disparity, by sending $x\in V^\mathbb{P}(p)$ to its naturally corresponding Kripke sets $\tau_{pq}(x) \in V^\mathbb{P}(q)$. 
Now we can describe the idea `$x(q)$ is larger than $x(p)$' in terms of transition functions as $\tau^"_{pq}[x(p)]\subseteq x(q)$, which is the very statement of the monotonicity condition.

We need to justify that a sequence of classes satisfying \autoref{Definition:KripkeInformal} exists. With Powerset, we can easily see that not only $V^\mathbb{P}_\alpha(p)$ exists for each ordinal $\alpha$, but also each $V^\mathbb{P}_\alpha(p)$ is a set.
However, we often work over $\mathsf{CZF}^-$, and $V^\mathbb{P}_\alpha(p)$ need not be a set. Thus, we appeal to inductive definitions for the existence of the hierarchy we defined in \autoref{Definition:KripkeInformal}, which we will do by the following lemma:

\begin{lemma}
    There is a sequence of classes $\langle V^\mathbb{P}_\alpha(p) \mid \alpha\in\Ord \land p\in\mathbb{P} \rangle$ and a class function $\tau$ satisfying the conditions in \autoref{Definition:KripkeInformal}.
    Furthermore, if we have $\mathsf{Pow}$, then $V^\mathbb{P}_\alpha(p)$ is a set for each $\alpha$ and $p$.
\end{lemma}
\begin{proof}
    Let $\Phi$ be an inductive definition defined as follows: $\langle a,x\rangle \in\Phi$ if and only if
    $x$ is a function of domain $\uparrow p$ for some $p\in\mathbb{P}$ which satisfies the following conditions:
    \begin{enumerate}[label=(\roman*)]
        \item \label{Labeling:DefFormulaPhi1} $x(q)\subseteq a$ for every $q\ge p$,
        \item \label{Labeling:DefFormulaPhi2}
        For each $q\ge p$ and $y\in x(q)$, $y$ is a function of domain $\uparrow q$ and
        \item \label{Labeling:DefFormulaPhi3}
        If $p\le q\le r$ then $y\restricts (\uparrow r) \in x(r)$ for every $y\in x(q)$.
    \end{enumerate}
    By Class Inductive Definition Theorem,  \autoref{Theorem:ClassInductiveDefThm}, there is a smallest $\Phi$-closed class $I$ and a $\Gamma_\Phi$-hierarchy $\langle I^\alpha \mid \alpha\in\Ord\rangle$.
    Now define 
    \begin{itemize}
        \item $V^\mathbb{P}(p) = \{x\in I\mid \dom x=(\uparrow p)\}$,
        \item $V_\alpha^\mathbb{P}(p) = \{x\in I^\alpha\mid \dom x=(\uparrow p)\}$, and
        \item $\tau_{pq}(x) = x\restricts (\uparrow q)$ for $x\in V^\mathbb{P}(p)$.
    \end{itemize}
    
    We will show that the conditions in \autoref{Definition:KripkeInformal} hold for $V_\alpha^\mathbb{P}(p)$ and $\tau_{pq}$. Condition \ref{Labeling:SetKripke1}, \ref{Labeling:SetKripke2} and \ref{Labeling:SetKripke3} are easy to check. Hence let us check the condition \ref{Labeling:SetKripke4} holds.
    
    Assume inductively that \ref{Labeling:SetKripke4} holds for every $\beta\in\alpha$.
    For the one direction, let $x\in V^\mathbb{P}_\alpha(p)$, which is equivalent to $\dom x=\uparrow p$ and $x\in I^\alpha = \Gamma_\Phi(I^{\in\alpha})$.
    Hence, there is $a$ such that $a\subseteq I^{\in\alpha}$ and $\langle a,x\rangle \in \Phi$.
    Condition \ref{Labeling:DefFormulaPhi1} and \ref{Labeling:DefFormulaPhi2} of the defining formula of $\Phi$, combining with the inductive definition, implies $x(q)\subseteq \bigcup_{\beta\in\alpha} V^\mathbb{P}_\beta(q)$ for every $q\ge p$. 
    Moreover, \ref{Labeling:DefFormulaPhi3} is just another way to state $\tau_{qr}''[x(q)]\subseteq x(r)$ for every $p\le q\le r$.
    
    For the remaining direction, assume that $x$ is a function with the domain $\uparrow p$ which satisfies conditions in \ref{Labeling:SetKripke4}.
    We can see that the first condition implies \ref{Labeling:DefFormulaPhi1} and \ref{Labeling:DefFormulaPhi2}.
    We have observed that \ref{Labeling:DefFormulaPhi3} is equivalent to the second condition of \ref{Labeling:SetKripke4}.
    This completes the proof.
\end{proof}

The following lemma is useful to show a given function is a $\mathbb{P}$-name:
\begin{lemma}[The Closure lemma]\label{Lemma:KripkeNameLemma}
    Let $x$ be a function of domain $\uparrow p$ which satisfies $x(q)\subseteq V^\mathbb{P}(q)$ for all $q\ge p$ and monotonocity condition: for any $p\le q\le r$, $\tau_{qr}^"[x(q)]\subseteq x(r)$.
    Then $x\in V^\mathbb{P}(p)$.
\end{lemma}
\begin{proof}
    Direct from the inductive definition of the universe of Kripke sets $V^\mathbb{P}$.
\end{proof}

The forcing relation $\VDash$ over $V^\mathbb{P}$ is defined inductively for formulas. In the following definition, assume that $x,y\in V^\mathbb{P}(p)$.
\begin{itemize}
    \item $p\VDash x\in y \iff x\in y(p)$,
    \item $p\VDash x = y \iff \text{For each } q\ge p, x(q) = y(q)$,
    \item $p\VDash \phi\land \psi \iff p\VDash \phi \text{ and } p\VDash\psi$,
    \item $p\VDash \phi\lor\psi \iff p\VDash\phi \text{ or } p\VDash\psi$,
    \item $p\VDash \phi\to\psi \iff \text{For each $q\ge p$, if $q\VDash \phi$
    then $q\VDash \psi$}$,
    \item $p\VDash \forall x\phi(x) \iff \text{For each $q\ge p$ and 
    $x\in V^\mathbb{P}(q)$, $q\VDash\phi(x)$}$, and
    \item $p\VDash \exists x \phi(x) \iff \text{There is $x\in V^\mathbb{P}(p)$ 
    such that $p\VDash \phi(x)$}$.
\end{itemize}
We assume that every parameter in this definition is a member of $V^\mathbb{P}(p)$.
In many cases, however, we can face parameters that appear in a given formula which is not a member of $V^\mathbb{P}(p)$, but a member of $V^\mathbb{P}(r)$ for some $r\le p$.
In this case, we take an appropriate transition function to the parameters, so the parameters  belong to $V^\mathbb{P}(p)$. For example, if $r,s\le p$, $x\in V^\mathbb{P}(r)$ and $y\in V^\mathbb{P}(s)$ then $p\VDash\phi(x,y)$ is \emph{defined} by $p\VDash \phi(\tau_{rp}(x),\tau_{sp}(y))$.

The following lemma states $\in$ and $=$ are persistent up to the transition map. It also justifies the mentioned convention for parameters:
\begin{lemma}
The transition function respects = and $\in$. Formally, for each $x,y\in V^\mathbb{P}(p)$, if $p\VDash x=y$ or $p\VDash x\in y$ then $q\VDash \tau_{pq}(x) = \tau_{pq}(y)$ or $q\VDash \tau_{pq}(x)\in\tau_{pq}(y)$ respectively.
\end{lemma}
\begin{proof}
Direct from the definition of $=$, $\in$ of the Kripke language, and the monotonicity condition for Kripke sets.
\end{proof}

Observe that interpreting $\to$ and $\forall$ introduces a new frame variable. We will handle a series of universal quantifiers and a combination of $\forall$ and $\to$, so there are possibilities for introducing lots of frame variables.
Fortunately, we can prove that introducing new frame variables is unnecessary in these special cases:
\begin{lemma}\label{Lemma:QuantifierWithConditions} \pushQED{\qed}
    Let $p\in\mathbb{P}$ and $y_0,\cdots, y_m\in V^\mathbb{P}(p)$ be parameters of a formula $\phi$.
    \begin{enumerate}
        \item $p\VDash \forall x_0\cdots \forall x_n \phi(x_0,\cdots, x_n,y_0,\cdots y_m)$
        if and only if for each $q\ge p$ and $x_0,\cdots,x_n\in V^\mathbb{P}(q)$ we have $q\VDash \phi(x_0,\cdots, x_n,\tau_{pq}(y_0),\cdots, \tau_{pq}(y_m))$.
        
        \item $p\VDash \forall x[\phi(x,y_0,\cdots, y_m)\to\psi(x,y_0,\cdots, y_m)]$ if and only if for each $q\ge p$ and $x\in V^\mathbb{P}(q)$, we have
    	\begin{equation*}
    		q\VDash \phi(x,\tau_{pq}(y_0),\cdots, \tau_{pq}(y_m)) \implies q\VDash \psi(x,\tau_{pq}(y_0),\cdots, \tau_{pq}(y_m)). \qedhere 
    	\end{equation*}
    \end{enumerate}
\end{lemma}
A proof of the previous lemma follows from a direct calculation, so we omit it.
The next lemma states the forcing relation $p\VDash \phi(x)$ is $\Delta_0$ when $\phi$ is $\Delta_0$. This fact is useful when verifying $V^\mathbb{P}$ satisfies $\Delta_0$-separation if our $V$ only satisfies $\Delta_0$-Separation. Its proof follows from direct calculation, so we omit it.
\begin{lemma}\pushQED{\qed} 
    If $q\in\mathbb{P}$ and $x\in V^\mathbb{P}(q)$ then
    \begin{itemize}
        \item $q\VDash \forall y\in x\phi(y) \iff \text{For each $r\ge q$ and 
        $y\in x(r)$, $r\VDash \phi(y)$}$ and
        \item $q\VDash \exists y\in x \phi(x) \iff \text{For some $y\in x(q)$, 
        $q \VDash \phi(y)$}$. 
    \end{itemize}
    In particular, if $\phi$ is a $\Delta_0$-formula then $p\VDash \phi$ is also $\Delta_0$.
    \qedhere
\end{lemma}\popQED

We will show that $V^\mathbb{P}$ is a model of $\mathsf{CZF}$. Before checking that, we have to verify that $V^\mathbb{P}$ satisfies intuitionistic logic. 
In addition, we also need to verify the equality axioms are valid in $V^\mathbb{P}$: 
\begin{proposition}
$V^\mathbb{P}$ satisfies intuitionistic logic and equality axioms.
\end{proposition}
\begin{proof}
	The proof for intuitionistic logic is similar to \cite[Ch. 2, \S5, Theorem 5.10]{TroelstraDalen}, so we omit its proof.
	Equality axioms directly follow from the definition of $=$ and $\in$, with \autoref{Lemma:QuantifierWithConditions}. (See \cite[Lemma 2.0.2]{Lubarsky2005} for the list of equality axioms.)
\end{proof}

We can construct some canonical names from given names. We will use the following names to show that the axioms of $\mathsf{CZF}$ are valid in $V^\mathbb{P}$:
\begin{definition}
Let $p\in\mathbb{P}$ and $x,y,y_0,\cdots,y_n\in V^\mathbb{P}(p)$. 
Define $\mathbb{P}$-names $\up(x,y)$, $\op(x,y)$, $\Union(x)$, $\Power(x)$, $\Sep_\phi(x;y_0,\cdots,y_n)$ with domain $\uparrow p$ as follows:
\begin{itemize}
    \item $\up(x,y)(q) = \{\tau_{pq}(x), \tau_{pq}(y)\}$,
    \item $\op(x,y)(q) = \up(\up(x,x)(q),\up(x,y)(q))(q)$,
    \item $\Union(x)(q) = \bigcup \{z(q) \mid z\in x(q)\}$,
    \item $\Power(x)(q) = \{z \in V^\mathbb{P}(q) \mid \forall r\ge q : z(r)\subseteq y(q)\}$ and
    \item $\Sep_\phi(x;y_0,\cdots,y_n)(q) = \{z\in x(q) \mid q\VDash \phi(z, \tau_{pq}(y_0), \cdots, \tau_{pq}(y_n) )\}$
\end{itemize}
\end{definition}
Note that $\Power(x)$ is a set only when Powerset is valid in $V$. Thus, we refer to $\Power(x)$ only when Powerset is valid.
Like canonical names for sets in $V$ under a classical forcing, we can define canonical $\mathbb{P}$-names for sets in $V$:
\begin{definition}
Let $x$ be a set. Define its canonical name $\check{x}$ with domain $\mathbb{P}$ as follows:
	\begin{equation*}
		\check{x}(p) = \{\tau_{\bot p}(\check{y}) \mid y\in x\}.
	\end{equation*}
\end{definition}
We can see that $\VDash \text{``$\check{\alpha}$ is an ordinal''}$ if $\alpha$ is an ordinal. In particular, we can see that $\check{\omega}$ witnesses Infinity.

\begin{theorem}\label{Theorem:KripkeModelsCZF}
Every axiom of $\mathsf{CZF}^-$ is valid in $V^\mathbb{P}$. 
Moreover, if $V$ satisfies either  Subset Collection, $\mathsf{Sep}$ or $\mathsf{Pow}$ then $V^\mathbb{P}$ also satisfies Subset collection, $\mathsf{Sep}$ or $\mathsf{Pow}$ respectively.
\end{theorem}
\begin{proof}
\begin{enumerate}[label=(\roman*)]
    \item Extensionality: This directly follows from the definition of $\in$ and $=$ of Kripke models.
    
    \item Pairing: We will see that $\up(x,y)$ witnesses Pairing for each $x,y\in V^\mathbb{P}(p)$. We have to check the following holds:
    \begin{equation*}
        p\VDash \forall z [z\in\up(x,y)\leftrightarrow (z=x\lor z=y)].
    \end{equation*}
    In one direction, take $q\ge p$, $z\in V^\mathbb{P}(q)$ and suppose that $r\VDash z\in \up(x,y)$ holds for $r\ge q$. 
    Then $\tau_{qr}(z) \in \{\tau_{pr}(x),\tau_{pr}(y)\}$, so $\tau_{qr}(z) = \tau_{pr}(x)$ or $\tau_{qr}(z) = \tau_{pr}(y)$. 
    By the definition of $\tau$, it implies
    \begin{itemize}
        \item $z(s) = x(s)$ for every $s\ge r$ or
        \item $z(s) = y(s)$ for every $s\ge r$.
    \end{itemize}
    Therefore $r\VDash x=z\lor y=z$. The other direction follows from reversing our previous proof.
    
    \item Union: We can prove that $\Union(x)$ witnesses Union. The details are left to the readers.
    
    \item $\in$-Induction: We have to show that 
	\begin{equation*}
		\bot \VDash \forall x(\forall y\in x \phi(y)\to\phi(x)) \to \phi(x).
	\end{equation*}
    Suppose that 
    \begin{equation}\label{Formula:InInductionKrikpe00}
        p\VDash \forall x (\forall y\in x\phi(y)\to\phi(x))
    \end{equation}
    holds. 
    We shall verify $p\VDash \forall x\phi(x)$, which is equivalent to
    \begin{equation}\label{Formula:InInductionKrikpe01}
        \forall \beta\in\Ord\bigl[\forall q\ge p \forall x\in V^\mathbb{P}_\beta(q)[q\VDash \phi(x)]\bigr].
    \end{equation}
    Note that the equivalence follows from the tautology 
    $[(\exists x\phi(x)) \to \psi]\leftrightarrow [\forall x (\phi(x)\to\psi)]$
    of intuitionistic logic, where $x$ does not occur free in $\psi$.

    Now assume inductively that \eqref{Formula:InInductionKrikpe01} holds for every $\beta\in \alpha$.
    Fix $q\ge p$ and let $x\in V^\mathbb{P}_\alpha(q)$. If $y\in x(r)$ for $r\ge q$, then $y\in V^\mathbb{P}_\beta(r)$ for some $\beta\in\alpha$.
    Therefore $r\VDash \phi(y)$ for all $r\ge q$ and $y\in x(r)$. This implies $q\VDash \forall y\in x \phi(y)$. 
    By \eqref{Formula:InInductionKrikpe00}, $q\VDash \forall y\in x \phi(y)$ implies $q\VDash \phi(x)$.
    By induction on ordinals over the ground model $V$, we have \eqref{Formula:InInductionKrikpe01} for every $\beta\in \Ord$.
    
    \item Infinity: We shall prove that $\check{\omega}$ witnesses Infinity. Formally, $\check{\omega}$ is closed under successor operator $x\mapsto x\cup\{x\}$ internally. 
    By induction on $n$, we can prove
    \begin{equation*}
    	\bot\VDash \widecheck{(n+1)} = \check{n}\cup \{\check{n}\}.
    \end{equation*}
    Moreover, $\check{\omega}(p) = \{\tau_{\bot p}(\check{n}) \mid n\in\omega\}$.
    Therefore, if $p\VDash x\in\check{\omega}$, which implies $x=\tau_{\bot p}(\check{n})$ for some $n\in\omega$, then we have $p\VDash x\cup\{x\} \in \check{\omega}$.
    
    \item ($\Delta_0$-)Separation: Let $x,y\in V^\mathbb{P}(p)$ be Kripke names.
    Since our ground model satisfies $\Delta_0$-separation,
    $\Sep_\phi(x;y_0,\cdots,y_n)$ is a set for a bounded formula 
    $\phi(x,y_0,\cdots,y_n)$. We can easily see that
    \begin{equation*}
        p \VDash \forall z [z\in \Sep_\phi(x;y_0,\cdots, y_n) 
        \leftrightarrow z\in x\land \phi(z,y_0,\cdots, y_n)]
    \end{equation*}
    by the definition of $\Sep_\phi(x;y_0,\cdots y_n)$.
    The same proof can be applied to show $V^\mathbb{P}$ satisfies Full separation in the case when $V$ satisfies Full separation.

    \item Strong Collection:
    Fix $p\in\mathbb{P}$ and $a\in V^\mathbb{P}(p)$. Suppose that
    $q\ge p$ satisfies $q\VDash \forall x\in a \exists y \phi(x,y)$, 
    which is equivalent to
    \begin{equation}\label{Formula:StrCollKripke00}
        \forall r\ge q\forall x\in a(r) \exists y \bigl[ y\in V^\mathbb{P}(r) \land
        r\VDash \phi(x,y)\bigr].
    \end{equation}
    Define $A=\bigcup_{r\ge q}\{r\}\times a(r)$. Then 
    \eqref{Formula:StrCollKripke00} can be restated as
    \begin{equation}\label{Formula:StrCollKripke01}
        \forall \pi \in A\exists y [y \in V^\mathbb{P}(\jmath_0\pi) 
        \land \jmath_0\pi\VDash \phi(\jmath_1\pi, y)].
    \end{equation}
    (where $\jmath_0$ and $\jmath_1$ are canonical projections of a  pairing.)
    By Strong Collection in $V$, we can find a set $C$ such that
    \begin{equation}\label{Formula:StrCollKripke02}
        \forall \pi \in A\exists y\in C [y \in V^\mathbb{P}(\jmath_0\pi) 
        \land \jmath_0\pi\VDash \phi(\jmath_1\pi, y)]
    \end{equation}
    and
    \begin{equation}\label{Formula:StrCollKripke03}
        \forall  y\in C\exists \pi \in A [y \in V^\mathbb{P}(\jmath_0\pi) 
        \land \jmath_0\pi\VDash \phi(\jmath_1\pi, y)].
    \end{equation}
    $C$ itself is not a $\mathbb{P}$-name, but a collection of $\mathbb{P}$-names. Now let us turn $C$ to a $\mathbb{P}$-name witnessing the instance of Strong Collection.
    Let us define $b\in V^\mathbb{P}(q)$ as follows:
    \begin{equation}\label{Formula:StrCollInstanceB}
        b(r) = \{\tau_{sr}(y) \mid y\in C\land q\le s\le r \land 
        \dom y = \uparrow s \}.\}
    \end{equation}
    $b$ is well-defined by Replacement and $\Delta_0$-Separation.
    We will show that $b\in V^\mathbb{P}(q)$ and $b$ witnesses Strong Collection.
    Note that we do not need the full $b(r)$ as an instance of a collection set. In fact, a subset $\{y\in C \mid \dom y = \uparrow s\}$ suffices to witness Strong Collection. However, other $\mathbb{P}$-names are essential for the monotonicity of $b$.
    
    We can easily show $b\in V^\mathbb{P}(q)$ by checking conditions in \autoref{Lemma:KripkeNameLemma}.
    It remains to show that $b$ witnesses Strong Collection. We shall prove
    \begin{equation*}
        q \VDash \forall x\in a \exists y\in b \phi(x,y) \land
        \forall y\in b \exists x\in a \phi(x,y).
    \end{equation*}
    However, this is a direct corollary of \eqref{Formula:StrCollKripke02},
    \eqref{Formula:StrCollKripke03} and the definition of $b$.

	\item Subset Collection: We will show that $V^\mathbb{P}$ validates Fullness if $V$ satisfies Subset Collection. We shall prove
	\begin{equation*}
		\bot\VDash \forall a, b\exists c [\forall x\colon a\rightrightarrows b
		\exists y\in c [y\colon a\rightrightarrows b\land\, y\subseteq x]].
	\end{equation*}
    Fix $p\in\mathbb{P}$ and $a,b\in V^\mathbb{P}(p)$. Notice that
    for $q\ge p$ and $x\in V^\mathbb{P}(q)$, we have
    \begin{align}
        q\VDash x\colon a\rightrightarrows b 
        &\iff 
        q\VDash \forall u\in a\exists v\in b [ \op(u,v)\in x] \\
        &\iff \forall r\ge q \forall u\in a(r) \exists v\in b(r)[ \op(u,v)\in x(r)].
    \end{align}
    For each $q\ge p$, let $A_q=\bigcup_{r\ge q} \{r\}\times a(r)$ and 
    $B = \bigcup_{r\ge p} b(r)$.
    Now consider the following relation with the parameter $x$:
    \begin{equation*}
        R_x = \{\langle \langle r,u\rangle, v\rangle \in A_q\times B \mid
        \dom v = \uparrow r\land \op(u,v) \in x(r).\}
    \end{equation*}
    We can see that $R_x\subseteq A_q\times B$ and $R_x\colon  A_q\rightrightarrows B$.
	Now apply Subset Collection to $\mathcal{A}^{A_q, B}(R_x)$, where
	$\mathcal{A}^{A_q, B}$ is defined in \autoref{Lemma:PrelimAdjuectmentFtn}. Then we have a family $\mathcal{F}$ of sets, depending on $q\ge p$, such that 
	\begin{equation*}
		\forall x \bigl[\bigl(\mathcal{A}^{A_q, B}(R_x)\colon A_q\rightrightarrows A_q\times B \bigr)
		\to \bigl(\exists C\in \mathcal{F} (\mathcal{A}^{A_q, B}(R_x) \colon A_q \lrlrarrows C)\bigr)\bigr].
	\end{equation*}
	By \autoref{Lemma:PrelimAdjuectmentFtn}, this is equivalent to
	\begin{equation}\label{Formula:SubCollKripke00}
		\forall x [ (R_x\colon A_q\rightrightarrows B) \to \exists C\in \mathcal{F} \colon C\subseteq R_x \land C\colon A_q\rightrightarrows B].
	\end{equation}
	In sum, we have for each $q\ge p$ there is $\mathcal{F}$ satisfying \eqref{Formula:SubCollKripke00}.
	Then apply Collection to the previous claim to get a set $\mathbb{C}$, which satisfies for each $q\ge p$ there is $\mathcal{F}\in\mathbb{C}$ such that \eqref{Formula:SubCollKripke00} holds.
    
    Now let us start the construction of a $\mathbb{P}$-name witnessing Fullness. 
    We first define a classification of relations given as follows: 
    \begin{equation}\label{Formula:SubCollKripke01}
        \mathcal{Z}_q = \{\mathcal{F}\cap \operatorname{mv}(A_q,B) \mid \mathcal{F}\in\mathbb{C}\}.
    \end{equation}
	Since the formula ``$R\in \mv(A_q,B)$'' is $\Delta_0$, $\mathcal{F}\cap\mv(A_q,B)$ is a set. Hence $\mathcal{Z}_q$ is a set for each $q\ge p$.
	
    Now for each $C\in \mathcal{Z}_q$, define
    \begin{equation}\label{Formula:SubCollKripke02}
        z_{q,C}(r) = \{ \tau_{sr}(\op(u,v)) \mid q\le s\le r \land 
        \langle \langle s,u\rangle, v\rangle \in C \land \dom v = \uparrow s\}.
    \end{equation}
    $z_{q,C}$ itself does \emph{not} witness Fullness, but is an element of a
    name witnessing Fullness. Note that $z_{q,C}\in V^\mathbb{P}(q)$ by
    \autoref{Lemma:KripkeNameLemma}.
    Now let us define a name $c$ with domain $\uparrow p$ as follows:
    \begin{equation*}
        c(q) = \{\tau_{sq}(z_{s,C}) \mid p\le s\le q \land C\in \mathcal{Z}_q\}.
    \end{equation*}
    We can see that $c$ is a name by \autoref{Lemma:KripkeNameLemma}.
    
    It remains to show that $c$ witnesses Fullness:
    For given $x\in V^\mathbb{P}(q)$ satisfying $q\VDash x\colon a\rightrightarrows b$, we can find $\mathcal{F}\in\mathbb{C}$ and $C\in\mathcal{F}$ such that the consequent of \eqref{Formula:SubCollKripke00} holds.
    By definition of $z_{q,C}$ and $C\subseteq R_x$, we have $q\VDash z_{q,C}\subseteq x$. Since $z_{q,C}\in c(q)$, $c$ witnesses Fullness.
    
	\item Powerset: Let us assume Powerset over $V$. We can show that $\Power(x)$ exists for $x\in V^\mathbb{P}$.  Moreover, we can prove that $\Power(x)$ witnesses Powerset. The full proof is left to the readers, but note that the following equivalence is useful:
	\begin{equation*}
		p\VDash y\subseteq x\iff \forall q\ge p : y(q)\subseteq x(q). \qedhere 
	\end{equation*}
\end{enumerate}
\end{proof}

\subsection{Kripke rank}
In this subsection, we define a rank for $\mathbb{P}$-names.
To define this, consider the following relation:
\begin{equation*}
	y\vartriangleleft x \iff \exists q\in \dom x [y\in x(q)].
\end{equation*}
We want to use an recursive definition on $\vartriangleleft$, which requires $\vartriangleleft$ to be fully inductive:
\begin{proposition}
    $\vartriangleleft$ is fully inductive.
\end{proposition}
\begin{proof}
	$\vartriangleleft$ has a $\Delta_0$-definition. Moreover, $\ext_\vartriangleleft(x)=\bigcup \ran x$. Hence, it is sufficient to show that the $\vartriangleleft$-induction scheme is valid.
  	The idea of our proof --- showing the induction schema is valid by applying induction on the level of hierarchy --- will be akin to that of showing $\in$-induction over $V^\mathbb{P}$.
	
    Suppose that
    \begin{equation}\label{Formula:InductionOnTriangle}
        \forall x\in V^\mathbb{P}[ (\forall y\in V^\mathbb{P}: y\vartriangleleft x\to \phi(y))\to \phi(x)]
    \end{equation}
	holds. Furthermore, assume inductively that $\forall x\in V^\mathbb{P}_\beta\phi(x)$ holds for all $\beta\in\alpha$.
    
    If $x\in V^\mathbb{P}_\alpha$ and $y\vartriangleleft x$, then $y\in V^\mathbb{P}_\beta$ for some $\beta\in\alpha$.
    Hence, we have $\phi(y)$ by the inductive hypothesis. Therefore, we can derive $\phi(x)$ from \eqref{Formula:InductionOnTriangle}.
\end{proof}

We are ready to define a rank for Kripke sets: take recursively that
\begin{equation*}
    \krk x = \sup\{\krk y+1 \mid y\vartriangleleft x\}.
\end{equation*}
Note that $\krk x$ coincides with $\sup_{q\in \dom x}\sup \{\krk y+1\mid y\in x(q)\}$.

The following proposition shows our Kripke rank behaves like the rank for actual sets:
\begin{proposition}\label{Proposition:KripkeRankandNameRel}
    Let $x\in V^\mathbb{P}(p)$. Then we have $x(q)\subseteq V^\mathbb{P}_{\krk x}(q)$
    for all $q\ge p$.
\end{proposition}
\begin{proof}
	The proof uses induction on $\vartriangleleft$. Assume that the proposition holds for all $y\vartriangleleft x$, so we have
	\begin{equation*}
		\forall q\ge p\forall y\in x(q) \forall r\ge q: y(r)\subseteq V^\mathbb{P}_{\krk y}(r).
	\end{equation*}
	By the previous sentence and the conditions on $V^\mathbb{P}$ of  \ref{Labeling:SetKripke4}, we have $y\in V^\mathbb{P}_{\krk y+1}(q)\subseteq V^\mathbb{P}_{\krk x}(q)$ for all $y\in x(q)$.
	Hence $x(q)\subseteq V^\mathbb{P}_{\krk x}(q)$.
\end{proof}

Kripke rank can be changed if we change the domain of a given Kripke set by applying the transition functions. The following lemma shows that the Kripke rank decreases under the transition functions. Furthermore, if $\mathbb{P}$ is linear, then the Kripke rank is invariant under the transition functions.
\begin{lemma}\label{Lemma:KripkeRankAndRestriction}
    Let $x\in V^\mathbb{P}(p)$ and $p\le q$. Then $\krk x\ge \krk \tau_{pq}(x)$. 
    If $\mathbb{P}$ is linear then $\krk x = \krk \tau_{pq}(x)$.
\end{lemma}
\begin{proof}
    $\krk x\ge \krk \tau_{pq}(x)$ follows from that if $y\vartriangleleft \tau_{pq}(x)$ then $y\vartriangleleft x$. 
    Now suppose that $\mathbb{P}$ is linear. We will use the induction on
    $\vartriangleleft$: assume that $\krk y = \krk \tau_{pq}(y)$ holds
    for all $y\vartriangleleft x$.
    Then
    \begin{equation*}
    \begin{array}{lcl}
        \krk \tau_{pq}(x) &=& \sup \{\krk y + 1 \mid y\vartriangleleft x\} \\
        &=& \sup_{r\ge q}\sup \{\krk y + 1 \mid y\in x(r)\}.
    \end{array}
    \end{equation*}
    We want to show that the last one is greater than or equal to $\krk x=\sup_{r\ge p}\sup \{\krk y + 1 \mid y\in x(r)\}$.
    If $r\ge p$, then either $r\ge q$ or $q\ge r$. If $q\ge r$, then by the inductive assumption,
    \begin{equation*}
        \sup \{\krk y+1\mid y\in x(r)\} = \sup\{\krk \tau_{rq}(y) + 1\mid y\in x(r)\}.
    \end{equation*}
    and the latter is less than or equal to $\sup\{\krk y+1\mid y\in x(q)\}$ by the monotonicity of Kripke names. Thus 
    \begin{equation*}
        \sup_{p\le r\le q}\sup\{\krk y+1\mid y\in x(r)\}\le \sup\{\krk y+1\mid y\in x(q)\}
    \end{equation*}
    and we have $\krk x\le \krk \tau_{pq}(x)$.
\end{proof}

\subsection{Slicing function}
The construction of a model of $\mathsf{NDCom}$ that we will see needs examining the structure of Krikpe sets at a fixed stage. We will examine them by considering the membership relation at the given stage, formally given by
\begin{equation*}
	x\in_p y\qquad\text{if and only if}\qquad x\in y(p)
\end{equation*}
The idea yields the following definition of \emph{slicing function at stage $p$}:
\begin{definition}
	Let $x\in V^\mathbb{P}$ and $p\in \dom x$. Define a \emph{slicing function at stage $p$} $\mathbb{s}_p(x)$, is defined by
	\begin{equation*}
		\mathbb{s}_p(x) = \{\mathbb{s}_p(y) \mid y\in x(p)\}.
	\end{equation*}
\end{definition}
The definition of $\mathbb{s}_p$ uses recursion on $\in_p$.
If we work over $\mathsf{IZF}$, we only need to check that $\in_p$ is well-founded to justify our definition. We work over $\mathsf{CZF}^-$ in general, however, and it requires us to check that $\in_p$ is not just well-founded, but fully inductive:
\begin{lemma}[$\mathsf{CZF}^-$]
    $\in_p$ is fully inductive.
\end{lemma}
\begin{proof}
    The main idea of our proof is the same as that of showing the induction scheme for $\vartriangleleft$.
    It suffices to show that the $\in_p$-induction schema is valid.
    Suppose that
        \begin{equation}\label{Formula:InPInduction01}
            \forall y\in V^\mathbb{P}(p) : [y\in_p x \to \phi(y)]\to \phi(x)
        \end{equation}
    holds for all $x\in V^\mathbb{P}(p)$. Now we will show that
    \begin{equation}\label{Formula:InPInduction00}
        \forall \alpha\in \Ord \forall x\in V^\mathbb{P}_\alpha (p) : \phi(x)
    \end{equation}
    by induction on $\alpha$. 
    
    Suppose that $\forall x\in V^\mathbb{P}_\beta (p) : \phi(x)$ holds for all $\beta\in\alpha$.
    If $x\in V^\mathbb{P}_\alpha(p)$ and $y\in x(p)$, then $y\in V^\mathbb{P}_\beta (p)$ for some $\beta\in\alpha$. Therefore $\phi(y)$ holds for all $y\in_p x$. 
    By \eqref{Formula:InPInduction01}, $\phi(x)$ holds. Therefore $\phi(x)$ holds for all $x\in V^\mathbb{P}_\alpha(p)$.
    Hence \eqref{Formula:InPInduction00} follows by the induction on ordinals.
\end{proof}

\begin{lemma}[$\mathsf{CZF}^-$]\label{Lemma:PropertyofCollFtn} 
    Let $\mathbb{P}$ be a frame, $p\in\mathbb{P}$ and $x,y\in V^\mathbb{P}(p)$.
    \begin{enumerate}
        \item $x\in y(p)$ implies $\mathbb{s}_p(x)\in\mathbb{s}_p(y)$.
        \item $\mathbb{s}_p(x) = \mathbb{s}_p''[x(p)]$.
        \item $\mathbb{s}_p(\up(x,y)) = \{\mathbb{s}_p(x), \mathbb{s}_p(y)\}$.
        \item $\mathbb{s}_p(\Union(x)) = \bigcup\mathbb{s}_p(x)$.
        \item $\mathbb{s}_p(\Power(x))\subseteq \mathcal{P}(\mathbb{s}_p(x))$.
    \end{enumerate}
\end{lemma}
\begin{proof}
The initial three statements directly follow from the definition of $\mathbb{s}_p$ and $\up$. Hence we only give proof for the last two statements:
    \begin{enumerate}
    \setcounter{enumi}{3}
        \item For the one inclusion, we have
        \begin{equation*}
        \begin{array}{lll}
            \mathbb{s}_p(\Union(x)) 
            &= \mathbb{s}_p''[{\small\bigcup} \{z(p) \mid z\in x(p)\}]
            &= {\small\bigcup} \{\mathbb{s}_p''[z(p)] \mid z\in x(p)\}\\
            &= {\small\bigcup}\{\mathbb{s}_p(z) \mid z\in x(p)\}
            &\subseteq {\small\bigcup}\mathbb{s}_p''[x(p)] \\
            &= {\small\bigcup}\mathbb{s}_p(x).&
        \end{array}
        \end{equation*}
        For the remaining inclusion, observe that
        $w\in \mathbb{s}_p(x)$ if and only if $w=\mathbb{s}_p(y)$ for some $y\in x(p)$. Hence 
        $\bigcup \mathbb{s}_p(x) \subseteq \bigcup\{\mathbb{s}_p(y) \mid y\in x(p)\}
        = \mathbb{s}_p(\Union(x))$.
        
        \item By a direct calculation, we have
        \begin{equation*}
        \begin{array}{ll}
            \mathbb{s}_p(\Power(x)) &= 
            \mathbb{s}_p''[\{z\in V^\mathbb{P}(p) \mid \forall q\ge p : 
            z(q)\subseteq x(q)\}] \\
            &= \{\mathbb{s}_p(z)\mid z\in V^\mathbb{P}(p) \text{ and }
            \forall q\ge p : z(q)\subseteq x(q)\}\\
            &\subseteq \{\mathbb{s}_p(z)\mid \mathbb{s}_p(z)\subseteq 
            \mathbb{s}_p(x)\}\\
            &\subseteq \mathcal{P}(\mathbb{s}_p(x))
        \end{array}
        \end{equation*}
        as $z(p)\subseteq x(p)$ implies $\mathbb{s}_p(z)\subseteq \mathbb{s}_p(x)$.
        \qedhere
    \end{enumerate}
\end{proof}

\section{Double Complement over Kripke models}
\label{Section:DComAndKripke}
\subsection{Kripke models and $\mathsf{DCom}$}
In this subsection, we prove the following result:
\begin{theorem}\label{Theorem:KripkeModelsDCom}
    Let $\mathbb{P}$ be a linear frame. If $V$ satisfies $\mathsf{ZFC}$ then $V^\mathbb{P}$ satisfies $\mathsf{DCom}$.
\end{theorem}

The following lemma has a critical role in the proof of \autoref{Theorem:KripkeModelsDCom}:
\begin{lemma}\label{Lemma:LinearKrikpeBound}
Let $\mathbb{P}$ be a linear order. If $p\in \mathbb{P}$, $z,x\in V^\mathbb{P}(p)$ and $p\VDash \lnot\lnot(z\in x)$ then $\krk z<\krk x$ holds.
\end{lemma}
\begin{proof}
    Let $\mathbb{P}$ be a linear order. 
    $p\VDash\lnot\lnot(z\in x)$ implies there is $q\ge p$ such that
    $\tau_{pq}(z)\in x(q)$. Thus $\krk \tau_{pq}(z)< \krk x$.
    By \autoref{Lemma:KripkeRankAndRestriction}, 
    $\krk z = \krk\tau_{pq}(z)$, so we have the desired result.
\end{proof}

\begin{proof}[Proof of \autoref{Theorem:KripkeModelsDCom}] 
Let $x\in V^\mathbb{P}(p)$. Define a function $y$ of domain $\uparrow p$ as
    \begin{equation*}
        y(q) =\{z\in V_{\krk x}^\mathbb{P}(q) \mid \exists s\ge q : \tau_{qs}(z)\in x(s)\}.
    \end{equation*}
We can show $y\in V^\mathbb{P}(p)$ by \autoref{Lemma:KripkeNameLemma}. 
We claim that $y$ is a double complement of $x$: that is,
    \begin{equation*}
    	p\VDash \forall z : \lnot\lnot (z\in x)\leftrightarrow z\in y.
    \end{equation*}
	Let $q\ge p$ and $z\in V^\mathbb{P}(q)$. For the one direction, assume that $r\VDash \lnot\lnot( z\in x)$ holds for some $r\ge q$. By \autoref{Lemma:LinearKrikpeBound}, we have $\krk z<\krk x$. Thus we have $z\in V^\mathbb{P}_{\krk x}(q)$ by \autoref{Proposition:KripkeRankandNameRel} and the Closure lemma (\autoref{Lemma:KripkeNameLemma}).
	Furthermore, $r\VDash \lnot\lnot( z\in x)$ implies the existence of $s\ge q$ such that $\tau_{qs}(z)\in x(s)$. Hence $z\in y(q)$.
	
	For the remaining direction, let $q\VDash z\in y$, so that $z\in y(q)$. By definition, we have $s\ge q$ such that $\tau_{qs}(z)\in x(s)$. Since $s$ is comparable with any element of $\mathbb{P}$ due to linearity, we have $\forall r\ge q \exists s\ge s: \tau_{qs}(z)\in x(s)$. Thus $q\VDash \lnot\lnot (z\in x)$.
\end{proof} 

Here, assuming Powerset is necessary to ensure each $V^\mathbb{P}_{\krk x}(q)$ is a set.
The previous theorem uses the bound of $\krk z$ for names $z$ satisfying $z\in x^\dcom$. This bound follows from the linearity of $\mathbb{P}$.

\subsection{Models of $\mathsf{IZF + NDCom}$}
\label{Subsection:IZFNDCom_Model}

In this section, we show that some Kripke frames do not preserve Double Complement even if we start with $V\vDash \mathsf{ZFC}$.
In fact, we will see that $\mathbb{P}=2^{<\omega}$, the complete binary tree, forces $\mathsf{NDCom}$. Throughout this subsection, we work over $\mathsf{ZFC}$.
We will denote members of $2^{<\omega}$ as finite binary sequences. $0^n$ is the sequence of length $n$ whose entries are all $0$, and for two $p,q\in 2^{<\omega}$, $p^\frown q$ is the new finite sequence obtained by concatenating the two. $\bot$ denotes the empty sequence.

\begin{definition}
    Let $C$ be the class of $V^{2^{<\omega}}$-names $x$ satisfying the following: For each $n<\omega$, there is $a \subseteq V_n$ such that $\tau_{\bot,0^n1}(x) = \tau_{\bot,0^n1}(\check{a})$.
\end{definition}
Informally speaking, $x$ is a name that can grow when we move through conditions with 0 only. When extending the condition adds 1 to the condition $p$, then $p$ thinks the name to be equal to the check name. 
Note that if $x\in C$, then $\bot \VDash \lnot \lnot (x\in \check{V}_\omega)$. 
\begin{theorem}
    $C$ is a proper class.
\end{theorem}
\begin{proof}
    We are going to construct a sequence of sets of $2^{<\omega}$-names $\langle C_\alpha \mid \alpha\in\Ord \rangle$ such that $C_\alpha\subseteq C$ and $|C_\alpha|=\beth_\alpha$ for every $\alpha$. 
    First, let us define $C_0\subseteq C$ as the set of names $x\in V^{2^{<\omega}}(\bot)$ satisfying the following: There is $n<\omega$ and $a\in V_{n+1}\setminus V_n$ such that
    \begin{itemize}
        \item $x(p) = \varnothing$ if $p\ngeq 0^n$, and
        \item $\tau_{\bot,0^n}(x) = \tau_{\bot,0^n}(\check{a})$.
    \end{itemize}
    $n$ and $a$ uniquely specifies $x$, and there are countably many such $x$. Hence, $|C_0| = \omega = \beth_0$.
    Now say two names $x,y\in V^{2^{<\omega}}(0^n)$ are \emph{equal near $0^\omega$} if there is $m\ge n$ such that $\tau_{\bot,0^m}(x)=\tau_{\bot,0^m}(y)$. It is clear that if $x,y\in C_0$, then $x=y$ iff $x$ and $y$ are equal near $0^\omega$. 

    Now, let us assume that we are given $C_\alpha$. We also inductively assume that for two $x,y\in C_\alpha$, $x=y$ iff $x$ and $y$ are equal near $0^\omega$.
    For $x\in V^{2^{<\omega}}(p)$, let $f(x)\in V^{2^{<\omega}}(0^\frown p)$ be the new name recursively defined by
    \begin{equation*}
        f(x)(0^\frown q) = \{f(y) \mid y\in x(q)\} \text{ for }q\ge p.
    \end{equation*}
    One can see that $f(x)$ also satisfies the monotonicity condition, so $f(x)$ is also a $2^{<\omega}$-name.  Also, note that $x$ and $y$ are equal near $0^\omega$ iff $f(x)$ and $f(y)$ are equal near $0^\omega$.
    Given $C_\alpha$, let us define $C_{\alpha+1}$ to be the set of $z\in V^{2^{<\omega}}(\bot)$ such that there is $Z\subseteq C_\alpha$, we have
    \begin{itemize}
        \item $z(\bot)=\varnothing$.
        \item $z(1^\frown p) = \varnothing$ for every $p\in 2^{<\omega}$.
        \item $z(0^\frown p) = \bigl\{\tau_{0,0^\frown p}\bigl(f(x)\bigr) \bigm| x\in Z\bigr\}$.
    \end{itemize}

    We need the following lemma in a later calculation:
    \begin{lemma*}
        For a set $a$ and $p\in 2^{<\omega}$, we have $f\circ \tau_{\bot,p}(\check{a}) = \tau_{\bot,0^\frown p}(\check{a})$.
    \end{lemma*}
    \begin{proof}
        We prove it by induction on $a$: Suppose that for $b\in a$, we have $f\circ \tau_{\bot,p}(\check{b}) = \tau_{0,0^\frown p}(\check{b})$.
        Then for a given $q\ge p$,
        \begin{align*}
            f\circ\tau_{\bot,p}(\check{a}) (0^\frown q)  & = 
            \{f(w) \mid w\in \tau_{\bot,p}(\check{a})(q)\}
            =\{ f(w)\mid w\in \check{a}(q)\} \\
            &= \{f\circ \tau_{\bot,q}(\check{b}) \mid b\in a\} 
            = \{\tau_{\bot,0^\frown q}(\check{b}) \mid b\in a\} = \check{a}(0^\frown q)\\
            & = \tau_{\bot,0^\frown p}(\check{a})(0^\frown q).
        \end{align*}
        so we have $ f\circ\tau_{\bot,p}(\check{a}) = \tau_{\bot,0^\frown p}(\check{a})$.
    \end{proof}

    Now we claim that $C_{\alpha+1}\subseteq C$: Fix $Z\subseteq C_\alpha$ and consider the corresponding name $z$. We want to show that $z\in C$. To see this, we show for each $m<\omega$ there is $b\subseteq V_m$ such that $\tau_{\bot,0^m 1}(z) = \tau_{\bot,0^m1}(\check{a})$.
    We clearly have $\tau_{\bot,1}(z) = \tau_{\bot,1}(\check{\varnothing})$, so let us consider the case $m=n+1$.

    For each $y\in Z$ and $n<\omega$, let us fix $b_y \subseteq V_n$ such that $\tau_{\bot,0^n1}(y) = \tau_{\bot,0^n1}(\check{b}_y)$. Then we have for $p\ge 0^n 1$,
    \begin{align*}
        \tau_{0,0^{n+1}1}\bigl(f(y)\bigr) (0^\frown p) & = \{f(z)\mid z\in y(p)\} = \{f(z)\mid z\in \check{b}_y(p)\} \\
        &= \bigl\{ f\bigl( \tau_{\bot,p}(\check{c})\bigr) \bigm| c\in b_y\bigr\}  = \{\tau_{\bot,0^\frown p}(\check{c}) \mid c\in b_y\}
        \\ &= \check{b}_y(0^\frown p),
    \end{align*}
    so we have $\tau_{0,0^{n+1}1}\bigl(f(y)\bigr) = \tau_{\bot,0^{n+1}1}(\check{b}_y)$. This implies for $p\ge 0^n 1$, $\tau_{0,0^\frown p}\bigl(f(y)\bigr) = \tau_{\bot,0^\frown p}(\check{b}_y)$.
    
    If we take $a = \{b_y\mid y\in Z\}$, then $a\subseteq V_{n+1}$ and for $p\ge 0^n1$,
    \begin{equation*}
        z(0^\frown p) = \bigl\{ \tau_{0,0^\frown p}\bigl(f(y)\bigr) \bigm| y\in Z\bigr\} 
        =  \bigl\{ \tau_{\bot,0^\frown p}(\check{b}_y) \bigm| y\in Z\bigr\} = \check{a}(0^\frown p).
    \end{equation*}
    This shows $\tau_{\bot,0^{n+1}1}(z) = \tau_{\bot,0^{n+1}1}(\check{a})$, as desired.

    Now we claim that if $z_0,z_1\in C_{\alpha+1}$ are different names, then $z_0$ and $z_1$ are different near $0^\omega$. Let $Z_0,Z_1\subseteq C_\alpha$ be the subsets corresponding to $z_0$ an $z_1$ respectively.
    Without loss of generality, let us assume that there is $x\in Z_1\setminus Z_0$.
    Since $x$ is not equal to any element of $Z_0$ near $0^\omega$, we have $z_0(0^n)\neq z_1(0^n)$ for all $n<\omega$, which implies the desired result. Clearly, $|C_{\alpha+1}| = 2^{|C_\alpha|} =\beth_{\alpha+1}$.

    For a limit $\alpha$, we take $C_\alpha = \bigcup_{\beta<\alpha} C_\beta$. Then $C_\alpha\subseteq C$ and two distinct names of $C_\alpha$ are not equal near $0^\omega$. Again, we have $|C_\alpha| = \beth_\alpha$ for a limit $\alpha$.
\end{proof}

Now we will use $C$ to show that $\bot$ does not force $\mathsf{DCom}$:
\begin{theorem}
    For every $n<\omega$, $0^n\nVDash \exists x \forall y [\lnot\lnot (y\in \check{V}_\omega) \to y\in x]$. 
\end{theorem}
\begin{proof}
    If $0^n\VDash \exists x \forall y [y\in x\leftrightarrow \lnot\lnot (y\in \check{V}_\omega)]$, then we can find $x\in V^{2^{<\omega}}(0^n)$ such that for every $p\ge 0^n$ and $y\in V^{2^{<\omega}}(p)$ such that $p\VDash \lnot\lnot (y\in \check{V}_\omega)$, we have $y\in x(p)$. However, we know that $p\VDash \lnot\lnot(y\in\check{V}_\omega)$ for every $y\in C$, so $C\subseteq x(p)$. This is impossible since $C$ is a proper class.
\end{proof}

We can apply a similar argument to any $p\in 2^{<\omega}$. That is, by replacing every argument for $0^n$ with that for $p^\frown 0^n$, we can see that $p$ does not force there is $x\in V^{2^{<\omega}}(p)$ such that $x$ contains the double complement of $\check{V}_\omega$. This shows the following:
\begin{theorem} \pushQED{\qed}
    $V^{2^{<\omega}}\vDash \mathsf{NDCom}$. In fact, $V^{2^{<\omega}}\vDash \text{``No set can be a double complement of $\check{V}_\omega$.''}$ \qedhere 
\end{theorem}

\section{Metamathematics of $\mathsf{ADCom}$}
\label{Section:ADCom}
Lubarsky \cite{Lubarsky2005} constructed a Kripke model over the frame $\mathbb{P}=\Ord$ to construct a model of $\mathsf{CZF}$+$\mathsf{Sep}$+ $\mathsf{\lnot Pow}$. We will call this model \emph{Lubarsky's first model}. 
One can prove that Lubarsky's first model also satisfies $\mathsf{ADCom}$, and this is the reason why we explain the whole construction of the model.

Lubarsky mentioned in \cite{Lubarsky2005} that he could construct a similar model by choosing $\mathbb{P}=\omega$ alternatively and requiring Kripke sets to be eventually constant. He also stated that this approach could be construed as `taking a cofinal $\omega$-sequence through $\Ord$ and cutting the full model down to those nodes.%
\footnote{We can find this description from the arXiv version of \cite{Lubarsky2005} and not from the journal version.}' He does not construct his Kripke model in this way, as he expected neither of the two constructions is essentially harder than the other.
In this paper, however, we prefer the latter one (taking $\mathbb{P}=\omega$ and restricting Kripke sets to be hereditarily eventually constant) over the former one (namely, assuming $\mathbb{P}=\Ord$) to get a better consistency strength result. We want to require $\mathbb{P}$ a linearly ordered class, but $\Ord$ is not provably linearly ordered over $\mathsf{CZF}$: Assuming $\Ord$ is a linearly ordered class yields $\mathsf{\Delta_0\mhyphen LEM}$, and we want to avoid it.

Consider the Kripke model $V^\omega$, whose frame is the set of natural numbers $\omega$ with the usual ordering $\le$. For each stage $p$ and $m$, we want to define a subclass $K^m(p)\subseteq V^\omega(p)$ satisfying
\begin{equation}\label{Formula:PropertyOnKmp}
    K^m(p) = \bigl\{x\in V^\omega(p)\bigm| \forall r, q\ge m\bigl[ r\ge q\to x(r) = \tau_{qr}''[x(q)] \bigr]\land
    \forall q\ge p\bigl[ x(q)\subseteq K^m(q)\bigr]\bigr\}.
\end{equation}
Informally, $K^m(p)$ is a class of Kripke names in $V^\omega(p)$ that does not grow up hereditarily after Stage $m$.
We need an inductive definition to provide a precise definition of $K^m(p)$:
\begin{definition}
    Let $\Phi_m$ be an inductive definition defined by, $\langle a,x\rangle\in\Phi_m$ if and only if
    \begin{enumerate}
        \item $x\in V^\omega$,
        \item $\forall p\in\dom x [x(p)\subseteq a]$ and
        \item \label{Item: Phi m for K - Condition 3} $\forall q\ge p\ge m \bigl[p,q\in\dom x\to x(q) = \tau_{pq}''[x(p)]\bigr]$.
    \end{enumerate}
    Note that \ref{Item: Phi m for K - Condition 3} is equivalent to the following:
    \begin{enumerate}[label = (\arabic*$'$), start=3]
        \item $\forall k\in\omega\Bigl[(\dom x=\uparrow k) \to 
        \forall p\ge\max\{m,k\}\bigl[x(p) = \tau''_{\max\{m,k\},p}[x(\max\{m,k\})]\bigr]\Bigr]$.
    \end{enumerate}
    By the Class Inductive Definition theorem, each $\Phi_m$ defines the least
    $\Gamma_{\Phi_m}$-closed class $K^m$. Now take 
    $K^m(p) = \{x\in K^m\mid \dom x=\uparrow p\}$. Finally, let 
    $K(p) = \bigcup_{m\in\omega}K^m(p)$.
\end{definition}

We may expect $K^m(p)$ to satisfy \eqref{Formula:PropertyOnKmp} and behave like the usual Kripke structures. The following sequence of lemmas proves our expectation:
\begin{lemma}[Closure lemma for $K$]\label{Lemma:ClosureLemmaForK}
    Let $x$ be a function whose domain is $\uparrow p$ for some $p\in\omega$.
    If $x$ satisfies the defining formula in \eqref{Formula:PropertyOnKmp}, then
    $x\in K^m(p)$.
\end{lemma}
\begin{proof}
	Suppose that $x(q)\subseteq K^m(q)$ for all $q\ge p$. Take $a=\bigcup_{q\ge p} x(q)$. Then we can see that $\langle a, x\rangle \in \Phi_m$. Since $a\subseteq K^m$, we have $x\in K^m(p)$.
\end{proof}
We cannot weaken the condition of the above lemma to $x(q)\subseteq K(q)$ for all $q\ge p$, and we will see a counterexample against this generalization in \autoref{Theorem:Lubersky1modelADCom}.
The following lemma ensures some basic properties of Kripke models are also valid over $K(p)$:
\begin{lemma}
    \begin{enumerate}
        \item $K(p)\subseteq V^\omega(p)$ for all $p\in\omega$.
        \item $K^m(p)\subseteq K^n(p)$ if $m\ge n$.
        \item The transition map $\tau_{pq}$ restricted over $K^m(p)$ is a
        map from $K^m(p)$ to $K^m(q)$.
        \item If $a$ is a set, then $\check{a}\in K^0(0)$.
    \end{enumerate}
\end{lemma}
\begin{proof}
	The first and second statements follow from the definition of $K(p)$.
	For the third statement, observe that $\tau_{pq}$ is just a restriction, so it does not change anything except for the domain of an input.
	We can show the last statement by applying set induction on $a$:
	Suppose that $\check{b}\in K^0(0)$ for all $b\in a$. Clearly, $\check{a}$ stops expanding hereditarily after stage 0. Hence, \autoref{Lemma:ClosureLemmaForK} guarantees $\check{a}\in K^0(0)$.
\end{proof}

The forcing relation $\VDash_K$ over $K$ is defined by the same definition of $\VDash$ with the following modifications:
\begin{itemize}
    \item $p\VDash_K \forall x\phi(x)$ if and only if $q\VDash_K \phi(x)$ for all $q\ge p$ and $x\in K(q)$.
    \item $p\VDash_K \exists x\phi(x)$ if and only if $p\VDash_K \phi(x)$ for some $x\in K(p)$.
\end{itemize}
We examine some structural properties of $K$ before describing the main result.
We have decomposed the usual Kripke universe $V^\mathbb{P}$ into a hierarchy $\langle V^\mathbb{P}_\alpha\mid\alpha\in\Ord\rangle$ indexed by ordinals.
Furthermore, the hierarchy satisfies a closure condition (i.e., Condition \ref{Labeling:SetKripke4} of \autoref{Definition:KripkeInformal}).
We also want to decompose $K$ similarly to apply induction on ordinals to prove statements about $K$. The natural way to decompose $K$ is taking $K^m_\alpha(p) := K^m(p)\cap V^\omega_\alpha(p)$. 
It is easy to see that their union is $K^m(p)$. The following lemma states $K^m_\alpha(p)$ satisfies an analogue of Condition \ref{Labeling:SetKripke4}:
\begin{lemma}\label{Lemma:SlicedClosureLemmaForK}
    $x\in K^m_\alpha(p)$ if and only if $x$ is a function of domain $\uparrow p$ 
    such that $x$ stops expanding hereditarily after stage $m$,
    satisfies monotonicity and
    $x(q)\subseteq \bigcup_{\beta\in\alpha} K^m_\beta(q)$ holds for all $q\ge p$.
\end{lemma}
\begin{proof}
    Let us decompose the sentence 
    $x(q)\subseteq \bigcup_{\beta\in\alpha} K^m_\alpha(q)$ into a conjunction of $x(q)\subseteq \bigcup_{\beta\in\alpha} V_\alpha(q)$ and $x(q)\subseteq K^m(q)$. Then the equivalence follows from \autoref{Lemma:ClosureLemmaForK} and Condition \ref{Labeling:SetKripke4} of \autoref{Definition:KripkeInformal}.
\end{proof}

Lubarsky observed that if a Kripke set $a\in K(m)$ does not change after stage $m$, then $\tau_{mp}$ behaves like an isomorphism for $p>m$ over $a$.
He also mentioned that the isomorphism over $a$ can be extended to the whole $K(m)$, by hereditarily translating the domain of Kripke names.
Lubarsky's observation is necessary to prove $K$ satisfies Separation. Thus, we provide it in a concrete form.

\begin{definition}
	The translation function $\mathbb{t}_s: V^\omega\to V^\omega$ is defined recursively as follows: let $p\in \omega$ and $x\in V^\omega(p)$.
	$\mathbb{t}_s(x)$ is a function of domain $\uparrow(p+s)$ such that
	\begin{equation*}
		\mathbb{t}_s(x)(q+s) = \mathbb{t}''_{s}[x(q)]
		= \{\mathbb{t}_s(z) \mid z\vartriangleleft x\land \dom z = \uparrow q\}
	\end{equation*}
	for all $q\ge p$.
\end{definition}
Note that $\mathbb{t}_s$ is defined on recursion on $\vartriangleleft$, so $\mathbb{t}_s$ is well-defined.

The translation function $\mathbb{t}_s$ `vertically translates' the domain of a Kripke set $x\in V^\omega(p)$ to $\uparrow (p+s)$ hereditarily. If $p\ge s$, we can imagine translating the domain of $x\in V^\omega(p)$ to $\uparrow(p-s)$.
We can see that $\vartriangleleft$ is fully inductive over $V^\omega(\ge p) := \bigcup_{q\ge p} V^\omega(q)$, so we can define a downward translation function:
\begin{definition}
	The downward translation function $\mathbb{t}_{-s} : V^\omega(\ge s)\to V^\omega$ is defined recursively as follows: Let $p\ge s$ and $x\in V^\omega(p)$. 
	Then $\mathbb{t}_{-s}$ is a function of domain $\uparrow(p-s)$ such that
	\begin{equation*}
		\mathbb{t}_{-s}(x)(q) = \mathbb{t}''_{-s}[x(q+s)]
		= \{\mathbb{t}_{-s}(z) \mid z\vartriangleleft x\land \dom z = \uparrow (q+s)\}
	\end{equation*}
	for all $q\ge p-s$.
\end{definition} 

The following lemma describes basic facts on the translation function.
\begin{lemma} \label{Lemma:translationftn00} Let $p,q\in\omega$.
\begin{enumerate}
	\item $\mathbb{t}_s$ is a bijection between $V^\omega(p)$ and $V^\omega(p+s)$.
	\item $\mathbb{t}_s(\tau_{pq}(x)) = \tau_{p+s,q+s}(\mathbb{t}_s(x))$ for every $x\in V^\omega(p)$ and $p\le q$.
	\item If $x\in K^m(p)$ for $m\ge p$ then $\mathbb{t}_s(x)\in K^{m+s}(p+s)$. 
	\item $\mathbb{t}_s$ is a bijection between $K(p)$ and $K(p+s)$.
\end{enumerate}
\end{lemma}
\begin{proof}\pushQED{\qed}
\begin{enumerate}
	\item In fact, $\mathbb{t}_s:V^\omega\to V^\omega(\ge s)$ and $\mathbb{t}_{-s}:V^\omega(\ge s)\to V^\omega$ are inverses of each other.
	It can be shown by induction on $\vartriangleleft$.
	
	We will show that $\mathbb{t}_{-s} \circ \mathbb{t}_s$ is the identity map. The other equality can be shown analogously.
	Assume inductively that $\mathbb{t}_{-s} \circ \mathbb{t}_s (y) = y$ holds for all $y\vartriangleleft x$.
	For $q\ge p$ and $x\in V^\omega(p)$, we have
	\begin{equation*}
	\mathbb{t}_{-s}\circ \mathbb{t}_s(x)(q) = \mathbb{t}_{-s}''[\mathbb{t}_s(x)(q+s)]
	= \mathbb{t}_{-s}''[\mathbb{t}_s'' [(x(q)]] = x(q)
	\end{equation*}
	since $y\vartriangleleft x$ for all $y\in x(q)$. Hence $\mathbb{t}_{-s}\circ \mathbb{t}_s(x) = x$ for all $x\in V^\omega$.
	Moreover, by examining the domains of Kripke sets, we can see that the restrictions $\mathbb{t}_s\restricts V^\omega(p)$ and $\mathbb{t}_{-s}\restricts V^\omega(p+s)$ are inverses of each other.
	
	\item Let $z\in \mathbb{t}_s(\tau_{pq}(x))(r+s)$ for $p\le q\le r$, so there is $w\in \tau_{pq}(x)(r) = x(r)$ such that $z=\mathbb{t}_s(w)$.
	Hence $z\in \mathbb{t}_s''[x(r)] = \mathbb{t}_s(x)(r+s) = \tau_{p+s,q+s}(\mathbb{t}_s(x))(r+s)$. This shows $\mathbb{t}_s(\tau_{pq}(x))(r+s)\subseteq \tau_{p+s,q+s}(\mathbb{t}_s(x))(r+s)$.
	Showing the reverse inclusion is analogous, so we omit it.
	
	\item If $x\in K^m(p)$, then $x(q)=\tau_{mq}''[x(p)]$ for every $q\ge m$.
	We can show $\mathbb{t}_s(x)(q+s) = \tau_{m+s,q+s}''[\mathbb{t}_s(x)(m+s)]$ for $q\ge m$, so $\mathbb{t}_s(x)\in K^{m+s}(p+s)$.
	
	\item We can show the following facts. The proof is similar to what we have already done, so we omit it.
	\begin{enumerate}
		\item $\mathbb{t}_{-s}(\tau_{p+s,q+s}(x)) = \tau_{p,q}(\mathbb{t}_{-s}(x))$ for all $x\in V^\omega(p+s)$ and $p\le q$. 
		\item If $x\in K^{m+s}(p+s)$ for $m\ge p$, then $\mathbb{t}_{-s}(x)\in K^m(p)$.
	\end{enumerate}
	Therefore, $\mathbb{t}_s: K(p)\to K(p+s)$ and $\mathbb{t}_{-s}: K(p+s)\to K(p)$.
	Since the composition of the two functions is the identity, they are inverses of each other.
	\qedhere
\end{enumerate}
\end{proof}

The following lemma shows that the transition function and the translation function coincide for Kripke sets that are constant through stages:
\begin{lemma}\label{Lemma:TranslationNTransition}
	Let $x\in K^p(p)$. Then $\mathbb{t}_s(x)=\tau_{p,p+s}(x)$.
\end{lemma}
\begin{proof}
	By the proof of 3 of \autoref{Lemma:translationftn00}, we have $\mathbb{t}_s(x)(q+s)=\tau_{p+s,q+s}''[\mathbb{t}_s(x)(p+s)]$ for $q\ge p$.
	Therefore, we can see that $\mathbb{t}_s(x)=\tau_{p,p+s}(x)$ is equivalent to $\mathbb{t}_s(x)(p+s)=x(p+s)$ for $x\in K^p(p)$.
	
	We will show that $\mathbb{t}_s(x)(p+s)=x(p+s)$ by induction: Suppose that $x\in K^p_\alpha(p)$ and the equality holds for every $y\in K^p_\beta(p)$ and $\beta\in\alpha$.
	We have $\mathbb{t}_s(y)=\tau_{p,p+s}(y)$ for every $y\in x(p)$.
	Therefore,
	$ \mathbb{t}_s(x)(p+s) = \mathbb{t}_s''[x(p)] = \tau_{p,p+s}''[x(p)] = x(p+s)$.
\end{proof}

The following lemma states $\mathbb{t}_s$ behaves like an isomorphism between $K(p)$ and $K(p+s)$:

\begin{lemma}\label{Lemma:TranslationIso} Let $p\in\omega$ and $\vec{a}\in K(p)$.
\begin{enumerate}
	\item If $p\VDash_K \phi(\vec{a})$, then $p+s\VDash_K \phi(\mathbb{t}_s(\vec{a}))$.
	\item If $p\ge s$ and $p\VDash_K \phi(\vec{a})$, then $p-s\VDash_K \phi(\mathbb{t}_{-s}(\vec{a}))$.
\end{enumerate}
\end{lemma}
\begin{proof}\pushQED{\qed}
	We will prove them simultaneously by induction on $\phi$, uniform to $p$, $s$ and $\vec{a}$.
	We will only consider $\mathbb{t}_{+s}$, as the case for $\mathbb{t}_{-s}$ is similar.
	\begin{enumerate}
		\item Atomic formulas: Suppose that $p\VDash_K a_0=a_1$, which is equivalent to $a_0=a_1$.
		Since $\mathbb{t}_{\pm s}$ are one-to-one, this is equivalent to $\mathbb{t}_s(a_0)=\mathbb{t}_{s}(a_1)$, so we have $p+ s \VDash_K \mathbb{t}_s(a_0) = \mathbb{t}_s(a_1)$.
		
		If $p\VDash_Ka_0\in a_1$, so $a_0\in a_1(p)$, then $\mathbb{t}_s(a_0)\in\mathbb{t}_s''[a_1(p)] = \mathbb{t}_s(a_1)(p+s)$.
		Thus $p+s\VDash_K \mathbb{t}_s(a_0)\in\mathbb{t}_s(a_1)$.
		
		\item Binary connections: Cases for $\land$ and $\lor$ are easy to prove, so we omit them. We will concentrate on the case $\to$.
		
		Suppose that $p\VDash_K (\phi\to \psi)(\vec{a})$ holds: That is, for every $q\ge p$, $q\VDash_K \phi(\tau_{pq}(\vec{a}))$ implies $q\VDash_K \psi(\tau_{pq}(\vec{a}))$.
		Furthermore, let $q+s\VDash_K \phi(\tau_{p+s,q+s}(\mathbb{t}_s(\vec{a})))$ holds for every $q\ge p$, 
		which is equivalent to $q+s\VDash_K \phi(\mathbb{t}_s(\tau_{p,q}(\vec{a})))$.
		By the inductive hypothesis on $\phi$, we have $q\VDash_K \phi(\mathbb{t}_{-s}(\mathbb{t}_s(\tau_{p,q}(\vec{a}))))$, so $q\VDash_K\phi(\tau_{p,q}(\vec{a}))$.
		Therefore, $q\VDash_K \psi(\tau_{pq}(\vec{a}))$ and we have
		$q+s\VDash_K \psi(\mathbb{t}_s(\tau_{pq}(\vec{a}))$ by the inductive hypothesis on $\psi$.
		
		In sum, $q+s\VDash_K \phi(\tau_{p+s,q+s}(\mathbb{t}_s(\vec{a})))$ implies 
		$q+s\VDash_K \psi(\tau_{p+s,q+s}(\mathbb{t}_s(\vec{a})))$ for all $q\ge p$, which means $p+s\VDash (\phi\to\psi)(\mathbb{t}_s(\vec{a}))$.

		\item Quantifiers: The case for $\exists$ is easy to check. For the case for $\forall$, suppose that $p\VDash_K \forall x\phi(x,\vec{a})$, which is equivalent to $q\VDash_K \phi(x,\tau_{pq}(\vec{a}))$ for every $q\ge p$ and $x\in K(q)$.
		Therefore, we have $q+s\VDash_K \phi(\mathbb{t}_s(x),\mathbb{t}_s(\tau_{pq}(\vec{a}))$ for every $x\in K(q)$.
		Since $\mathbb{t}_s$ is an onto function from $K(q)$ to $K(q+s)$, we have
		$q+s\VDash_K \phi(y,\tau_{p+s,q+s}(\mathbb{t}_s(\vec{a})))$ for every $y\in K(q+s)$.
		Since $q\ge p$ is arbitrary, we have $p+s\VDash_K\forall x \phi(x,\mathbb{t}_s(\vec{a}))$.
		\qedhere
	\end{enumerate}
\end{proof}

We are ready prove that $K$ satisfies $\mathsf{CZF}$:
\begin{theorem}
    $K$ satisfies $\mathsf{CZF}$. If the background theory satisfies Full Separation, then so does $K$.
\end{theorem}
\begin{proof}
    We can show Extensionality directly. 
	For $\in$-induction, we can use the argument which is described in the proof of 	$\in$-induction in \autoref{Theorem:KripkeModelsCZF}, with \autoref{Lemma:SlicedClosureLemmaForK}.
    
	The main idea of a proof for Pairing, Union, and Separation is the same:
    we shall prove that $K(p)$ is closed under operations $\Union$, $\up$ and $\Sep$. For $\up$, suppose that we have $x,y\in K(p)$. Without loss of generality, we can assume that there is $m\ge p$ such that $x,y\in K^m(p)$.
	We claim that $\up(x,y)\in K^m(p)$ by applying \autoref{Lemma:ClosureLemmaForK}: it is obvious that $\up(x,y)(q)\subseteq K^m(q)$ for all $q\ge p$.
	Moreover, if $q\ge m$ then
	\begin{equation*}
		\up(x,y)(q)=\tau_{pq}''\{x,y\}=\tau_{mq}''\{\tau_{pm}(x),\tau_{pm}(y)\} =\tau_{mq}''[\up(x,y)(m)].
	\end{equation*}
	Therefore, $\up(x,y)$ stops expanding hereditarily after stage $m$. The case for $\Union$ is similar, so we omit it. 
    
	For $\Sep$, it is obvious that $\Sep(x;\vec{y})\subseteq K^m(p)$ if $x,\vec{y}\in K^m(p)$.
	Without loss of generality, assume that $m\ge p$. To show $\Sep(x;\vec{y})$ stops expanding after stage $m$, we must check that $\Sep(x;\vec{y})(q)\subseteq\tau''_{mq}[\Sep(x;\vec{y})(m)]$ for every $q\ge m$. 
	
	Let $z\in \Sep(x;\vec{y})(q)$. Then $z\in x(q)$ and $q\VDash_K \phi(z,\tau_{pq}(\vec{y}))$.
	For notational convenience, let $s=q-m$.
	Since $q\VDash_K \phi(z,\tau_{pq}(\vec{y}))$, we have $m\VDash_K \phi(\mathbb{t}_{-s}(z),\mathbb{t}_{-s}(\tau_{pq}(\vec{y}))$.
	Observe that $\tau_{pm}(\vec{y})\in K^m(m)$, so we can apply \autoref{Lemma:TranslationNTransition} and we have
	\begin{equation*}
		\mathbb{t}_{-s}(\tau_{pq}(\vec{y})) = 
		\mathbb{t}_{-s}(\tau_{mq}(\tau_{pm}(\vec{y}))) = 
		\mathbb{t}_{-s}(\mathbb{t}_{s}(\tau_{pm}(\vec{y}))) = \tau_{pm}(\vec{y}).
	\end{equation*}
	Therefore, $m\VDash_K \phi(\mathbb{t}_{-s}(z),\tau_{pm}(\vec{y}))$.
	Moreover, $\mathbb{t}_s(\mathbb{t}_{-s}(z)) = z \in x(q) = \tau_{mq}''[x(m)] = \mathbb{t}_s''[x(m)]$. By injectivity of $\mathbb{t}_s$, we can conclude $\mathbb{t}_{-s}(z)\in x(m)$.
	Hence, we can conclude $\mathbb{t}_{-s}(z)\in \Sep(x;\vec{y})(m)$. 
	Since $z\in x(q)\subseteq K^m(q) = K^q(q)$, 
	$z \in \tau''_{mq}[\Sep(x;\vec{y})(m)]$.
	Note that our proof works not only for Bounded separation but also for Full separation if our background theory has Full separation.
    
    For Strong Collection and Subset Collection, we analyze the proof of \autoref{Theorem:KripkeModelsCZF} and modify it. 
    For Strong Collection, let $\phi(x,y)$ be a formula whose parameters belong to $K(p)$. Let $a\in K(p)$ and $q\ge p$ satisfies $q\VDash_K \forall x\in a\exists y\phi(x,y)$. 
	Take $m>q$ large enough so that $a$ and parameters of $\phi$ belong to $K^m(p)$.
	By the assumption on $\phi$, we have
    \begin{equation*}
        \forall r\ge q\forall x\in a(r) \exists y [y\in K(r) \land
        r\VDash_K \phi(x,y)].
    \end{equation*}
	By letting $A=\bigcup_{q\le r\le m}\{r\}\times a(r)$, we have the following analogue of \eqref{Formula:StrCollKripke01}:
    \begin{equation*}
	    \forall\pi\in A\exists y [y\in K(\jmath_0 \pi)\land \jmath_0\pi\VDash_K \phi(\jmath_1\pi,y)].
    \end{equation*}
	By Strong Collection, we can find $C$ satisfying the following statements, which are analogues of \eqref{Formula:StrCollKripke02} and \eqref{Formula:StrCollKripke03}:
	
    \begin{equation*}
        \forall \pi \in A\exists y\in C [y \in K(\jmath_0\pi) 
        \land \jmath_0\pi\VDash_K \phi(\jmath_1\pi, y)], \text{ and}
    \end{equation*}
    \begin{equation*}
        \forall  y\in C\exists \pi \in A [y \in K(\jmath_0\pi) 
        \land \jmath_0\pi\VDash_K \phi(\jmath_1\pi, y)].
    \end{equation*}

	We will define $b\in K(q)$ as we define $b$ in \eqref{Formula:StrCollInstanceB} for $r\le m$. For $r>m$, take $b(r) := \tau_{mr}''[b(m)]$. That is, $b$ is defined as follows:
    \begin{equation*}
    	b(r) = \begin{cases}
    	\{\tau_{sr}(y)\mid y\in C\land q\le s\le r\land \dom y=\uparrow s\} & 
    	\text{if $r\le m$},\\
    	\tau_{mr}'' [b(m)] & \text{if $r>m$}.
    	\end{cases}
    \end{equation*}
    
    We claim that $q\VDash \forall x\in a\exists y\in b \phi(x,y)$. 
	Let $x\in a(r)$ for $r\ge q$. If $r\le m$, then the proof in \autoref{Theorem:KripkeModelsCZF} works. Now assume that $r>m$. Then $x=\tau_{mr}(x_0)$ for some $x_0\in a(m)$. Hence there is $y_0\in b(m)$ such that $m\VDash_K \phi(x_0,y_0)$. Therefore, $r\VDash_K \phi(x,\tau_{mr}(y_0))$.
	This proves our desired result as $\tau_{mr}(y_0)\in b(r)$.
        
    It remains to show that the Axiom of Subset Collection is valid in $K$. Since Subset Collection is equivalent to Fullness under Strong Collection, it is sufficient to check that Fullness is valid in $K$.
    
    Let $a,b\in K(p)$. Take a large $m\ge p$ satisfying $a,b\in K^m(p)$, and follow the proof of \autoref{Theorem:KripkeModelsCZF} with some necessary 
    modifications. Then define $A_q$ and $B$ as follows:
    \begin{equation*}
        A_q=\bigcup_{q\le r\le m}\{r\}\times a(r) \quad\text{ and }\quad B=\bigcup_{p\le r\le m}b(r).
    \end{equation*}
    By Collection, we can find $\mathbb{C}$ such that for each $q$ satisfying
    $p\le q\le m$, there is $\mathcal{F}\in\mathbb{C}$ which satisfies
	\eqref{Formula:SubCollKripke00}.
	For each $p\le q\le m$, define $\mathcal{Z}_q$ as in \eqref{Formula:SubCollKripke01}. 
	For given $p\le q\le m$ and $C\in\mathcal{Z}_q$, define $z_{q,C}(r)$ same as \eqref{Formula:SubCollKripke02} if $q\le r\le m$.
	For $r>m$, define $z_{q,C}(r)=\tau''_{mr}[z_{q,C}(m)]$.	Then we can see that $z_{q,C}\in K^m(q)$.
    We finally let
    \begin{equation*}
	    c(q)=\begin{cases}
		\{\tau_{sq}(z_{s,C})\mid p\le s\le q\land C\in\mathcal{Z}_q\} &
		\text{if $q\le m$,} \\
		\tau_{mq}''[c(m)] & \text{if $q>m$.}
	 	\end{cases}
    \end{equation*}
    
    We claim that $c$ witnesses Fullness. We will show that
    \begin{equation*}
        p\VDash_K \forall x\colon a\rightrightarrows b \exists y\in c [ 
        (y\colon a\rightrightarrows b)\land y\subseteq x].
    \end{equation*}
    
    Let $q\ge p$, $x\in K(q)$ and $q\VDash_K x\colon a\rightrightarrows b$.
    Observe that $a$ and $b$ are not expanding after the stage $m$.
    Consider $x'$ of domain $\uparrow q$ given by 
    \begin{equation*}
        x'(r) = \begin{cases}
        x(r) & \text{if $r\le m$}, \\
        \tau_{mr}''[x(m)] & \text{if $r>m$.}
        \end{cases}
    \end{equation*}
	We will see that $x'$ is a `support' of $x$ that does not expand hereditarily after stage $m$.
	
	By definition, $x'$ does not grow hereditarily after stage $m$.  Furthermore, every elements of $x$ is not expanding hereditarily after stage $m$ as $q\VDash x\subseteq a\times b$, and so do elements of $x'$.
	Therefore, $x'\in K^m(q)$. Moreover, it is easy to see that $q\VDash_K x'\subseteq x$. Before considering $x'$ instead of $x$, we need to see that $q\VDash_K x'\colon a\rightrightarrows b$ holds. We shall prove the following sentence:
    \begin{equation*}
	    q\VDash_K \forall u\in a\exists v\in b (\op(u,v)\in x').
    \end{equation*}
	Let $r\ge q$ and $u\in a(r)$. If $r\le m$, then this follows from $q\VDash_K x\colon a\rightrightarrows b$. If $r>m$, then $u=\tau_{mr}(u_0)$ for some $u_0\in a(m)$, and we can find $v_0\in b(m)$ such that $\op(u_0,v_0)\in x(m)$.
	Therefore, $\op(u,\tau_{mr}(v_0)) = \tau_{mr}(\op(u_0,v_0))\in \tau_{mp}''[x(m)]=x'(r)$.
    
    Hence, there is $\mathcal{F}\in\mathbb{C}$ and $C\in\mathcal{F}$ satisfying \eqref{Formula:SubCollKripke00} for $x'$.
	It remains to show that $q\VDash z_{q,C}\subseteq x'$, but this follows from $C\subseteq R_x$ and the fact that both $z_{q,C}$ and $x'$ do not grow after stage $m$.
\end{proof}

\begin{theorem}\label{Theorem:Lubersky1modelADCom}
    $K$ validates $\mathsf{ADCom}$.
\end{theorem}
\begin{proof}
    Consider the following name:
    \begin{equation*}
        \check{1}_p(q) = \begin{cases} 
            \varnothing & \text{if $p<q$}\\
            \{\check{0}\} & \text{if $p\ge q$}.
        \end{cases}
    \end{equation*}
    We can see that $\check{1}_p\in K^p(0)$ and
    $0\VDash_K \lnot\lnot(\check{1}_p=\check{1})$ holds. We prove the latter one: 
    $0\VDash_K \check{1}_p\subseteq\check{1}$ is obvious. For the double negation
    of the remaining inclusion, observe that $0\VDash_K \lnot\lnot\phi$ if and only if
    there is $p$ such that $p\VDash_K \phi$ due to linearity of the Kripke frame.
    Thus $K$ thinks the double complement of $\{1\}$ contains all of
    $\check{1}_p$ for all $p\in\omega$, if it exists.
    
    However, if a Kripke set contains all of $\check{1}_p$, then
    it is a subset of none of $K^m$ for all $m\in\omega$. Hence no
    Kripke sets contain all of $\check{1}_p$. Therefore, no Kripke name is a 
    double complement of $\{1\}$. This proves $\WDec(1)$ does not exist, and
    so $\mathsf{ADCom}$ holds by \autoref{Theorem:ADComAndWDec}.
\end{proof}

Notice that the whole construction is conveyed over $\mathsf{CZF}$, so we can derive the following equiconsistency result:
\begin{corollary}\pushQED{\qed} 
$\mathsf{CZF}$ and $\mathsf{CZF+ADCom}$ are equiconsistent.
\qedhere
\end{corollary}\popQED

\section{Double Complement over realizability}
\label{Section:MRrealizability}
It is natural to ask what axioms and propositions are compatible with $\mathsf{DCom}$, $\mathsf{NDCom}$, and $\mathsf{ADCom}$. This subsection aims to establish the persistence of these principles under modest assumptions.
We will not devote ourselves to explaining basic facts and notations on realizability models. We will use notations and theorems that come from \cite{McCartyPhD} and \cite{Rathjen2003Realizability}. Therefore, readers who are not familiar with McCarty-style realizability need to consult the mentioned articles.
The following lemma shows our cumulative hierarchy on $\mathcal{A}$-names are closed
under the internal equality on $V(\mathcal{A})$. Its proof is available in \cite[Chapter 2, Lemma 6.2.]{McCartyPhD}:
\begin{lemma}[The closure lemma, $\mathsf{CZF}$] \label{Lemma:ClosureOrig} \pushQED{\qed} 
    Let $\mathcal{A}$ be a pca and $x,y\in V(\mathcal{A})$. Then the following holds for any ordinal $\alpha$:
    \begin{enumerate}
        \item If $V(\mathcal{A})\models z\in x$ and $x\in V(\mathcal{A})_\alpha$, 
        then there is $\beta<\alpha$ such that $z\in V(\mathcal{A})_\beta$.
        \item If $x\in V(\mathcal{A})_\alpha$ and $V(\mathcal{A})\models x=y$ then
        $y\in V(\mathcal{A})_\alpha$. \qedhere
    \end{enumerate}
\end{lemma}\popQED

The following lemma ensures we can apply separation for $\Vdash \phi(x)$ when $\phi(x)$ is a bounded formula. See \cite[Lemma 4.5.]{Rathjen2003Realizability} for its proof:
\begin{lemma}[$\mathsf{CZF}$]\label{Lemma:BoundedRealizeSeparation} \pushQED{\qed} 
	Let $\phi(x)$ be a bounded formula with parameters from $V(\mathcal{A})$.
	If $x\subseteq V(\mathcal{A})$ is a set, then
	\begin{equation*}
		\{\langle e,c\rangle : e\in \mathcal{A}\land c\in x\land e\Vdash \phi(c)\}
	\end{equation*}
	is a set.\popQED
\end{lemma}

We can prove that $\mathsf{DCom}$ is preserved under realizability by making use of the closure lemma. We need $\Sigma_1$-separation or Regular Extension Axiom ($\mathsf{REA}$) in our proof.
\begin{theorem}[$\mathsf{CZF + DCom + \Sigma\mhyphen Sep}$ or $\mathsf{CZF + DCom + REA}$]
\label{Theorem:PreservationDComOrig} Let $\mathcal{A}$ be a pca. Then $V(\mathcal{A})\models \mathsf{DCom}$.
\end{theorem}
\begin{proof}
    Let $x$ be an $\mathcal{A}$-name.
    By \autoref{Lemma:ClosureOrig}, we can find an ordinal $\alpha$ such that
    \begin{equation*}
        \forall z\in V(\mathcal{A}): V(\mathcal{A})\models z\in x 
        \implies z\in V(\mathcal{A})_\alpha.
    \end{equation*}
    Define $y=\{\mathbf{0}\} \times V(\mathcal{A})_\alpha^\dcom\cap V(\mathcal{A})$. 
    Before proceeding with the proof, we have to show that
	\begin{equation*}
		V(\mathcal{A})_\alpha^\dcom \cap V(\mathcal{A})
		= \{x \in V(\mathcal{A})_\alpha^\dcom \mid x\in V(\mathcal{A})\}
	\end{equation*} is a set. If $\mathsf{\Sigma\mhyphen Sep}$ holds, then we may use the fact that $x\in V(\mathcal{A})$ is a $\Sigma$-formula.
    We need the following fact when the case $\mathsf{REA}$ holds: if $B$ is a regular set, then $B\cap V(\mathcal{A})$ is a set. (See Lemma 6.1 of \cite{Rathjen2003Realizability} for its proof.) Take a regular set $B$ that contains $V(\mathcal{A})_\alpha^\dcom$. Then we have
    $V(\mathcal{A})_\alpha^\dcom \cap V(\mathcal{A}) = V(\mathcal{A})_\alpha^\dcom \cap B\cap V(\mathcal{A})$, so $V(\mathcal{A})_\alpha^\dcom \cap V(\mathcal{A})$ is a set.
    It remains to show that
    \begin{equation*}
        V(\mathcal{A})\models \forall z: \lnot\lnot (z\in x)\to z\in y.
    \end{equation*}
    Suppose that $e\Vdash \lnot\lnot (z\in x)$, which is equivalent to
    $\lnot\lnot(\exists f\in\mathcal{A} : f\Vdash z\in x)$. Since
    $\neg\neg(V(\mathcal{A})\models z\in x)$, we have $\neg\neg 
    (z\in V(\mathcal{A})_\alpha)$.
    Hence $z\in V(\mathcal{A})_\alpha^\dcom$. Therefore, we can see
    \begin{equation*}
        \mathbf{p00} \Vdash \lnot\lnot z\in y.
    \end{equation*}
    Now we can see that $\bblambda e. \mathbf{p00}$ is a realizer of $\mathsf{DCom}$.
\end{proof}
Note that we use $\mathsf{Pow}$ in this proof to ensure $V(\mathcal{A})_\alpha$ is a set. Since $\mathsf{CZF+DCom}$ implies $\mathsf{Pow}$, we can use the power sets.
The next theorem, which shows $\mathsf{NDCom}$ is absolute under McCarty-styled realizability, also uses $\mathsf{Pow}$. Unlike the previous result, we have to assume $\mathsf{Pow}$ separately:
\begin{theorem}[$\mathsf{CZF + Pow + NDCom}$]
    $V(\mathcal{A})\models \mathsf{NDCom}$.
\end{theorem}
\begin{proof}
	Let $a$ be an instance of $\mathsf{NDCom}$, so no set is a double complement of $a$. Now assume that $V(\mathcal{A})$ believes $\check{a}$ has a double complement: formally, there is $b\in V(\mathcal{A})$ such that
    \begin{equation*}
        V(\mathcal{A})\models \forall x:\lnot\lnot(x\in\check{a})\to x\in b.
    \end{equation*}
    Take $c=\{x\mid \exists e\in\mathcal{A}: e\VDash \check{x}\in b\}$. 
	We will prove that $c$ is a set by showing the following general statement: for each ordinal $\alpha$ and a set $x$, if $\check{x}\in V(\mathcal{A})_\alpha$ then $x\subseteq V_\alpha$.
	
	Its proof uses induction on $x$: suppose that we have $\forall \beta\in\Ord[\check{y}\in V(\mathcal{A})_\beta\to y\subseteq V_\beta]$ for all $y\in x$. Now assume that $\check{x}=\{\langle\mathbf{0},\check{y} \rangle\mid y\in x\}\in V(\mathcal{A})_\alpha$.
	Then for each $y\in x$, there is $\beta\in \alpha$ such that $\check{y}\in V(\mathcal{A})_\beta$. By the inductive assumption,  $y\subseteq V_\beta$ for some $\beta\in\alpha$. Hence $x \subseteq \bigcup_{\beta\in\alpha} \mathcal{P}(V_\beta) = V_\alpha$.
    
    Therefore, $c$ is equal to $\{x\in V_\alpha \mid \exists e\in\mathcal{A} (e\Vdash \check{x}\in b)\}$ for some $\alpha\in\Ord$. (Take $\alpha$ such that $b\in V(\mathcal{A})_\alpha$.) By \autoref{Lemma:BoundedRealizeSeparation}, $c$ is a set.
	Moreover, $c$ includes a double complement of $a$ by the following calculation:
    \begin{equation*}
    \begin{array}{lcl}
        \lnot\lnot(x\in a) 
        &\implies& \lnot\lnot(V(\mathcal{A})\models\check{x}\in\check{a})\\
        &\iff& V(\mathcal{A})\models \lnot\lnot(\check{x}\in\check{a})\\
        &\implies& V(\mathcal{A})\models \check{x}\in b\\
        &\implies& x\in c,
    \end{array}
    \end{equation*}
    which contradicts the assumption that $a$ does not have a double complement.
\end{proof}

It remains to show that $\mathsf{ADCom}$ is persistent under realizability.
Since $\mathsf{ADCom}$ is incompatible with $\mathsf{Pow}$, the persistency must not rely on power sets.
This proof mimics the proof of \cite[Lemma 6.22]{ZieglerPhD}.
\begin{theorem}[$\mathsf{CZF}$]
    If $\WDec(1)$ does not exist in $V$, then it also does not exist in $V(\mathcal{A})$.
    In other words, $\mathsf{ADCom}$ is absolute under realizability.
\end{theorem}
\begin{proof}
    Assume the contrary that $\WDec(1)$ does not exist in $V$, but $V(\mathcal{A})$ thinks $\WDec(1)$ exists. From $\WDec(1)=\{0\}\cup \{x\subseteq 1\mid \lnot\lnot(x=1)\}$, we can deduce that the existence of $\WDec(1)$ is equivalent to the existence of $\{x\subseteq 1\mid \lnot\lnot(x=1)\}$: one direction is trivial. For the remaining direction, use $\{x\subseteq 1\mid \lnot\lnot(x=1)\} = \WDec(1)\setminus\{0\}$.
    
    Therefore, there is $a\in V(\mathcal{A})$ and $e\in\mathcal{A}$ such that
    \begin{equation}\label{Formula:NotWDecPersistence00}
        e\Vdash\forall x: x\subseteq 1\land \lnot\lnot(x=1)\to x\in a.
    \end{equation}
    Now consider $W = \{ \{0\mid \text{$x$ is inhabited}\}\mid \exists f\in\mathcal{A}:
    \langle f,x\rangle\in a\}$.
    We claim that $\WDec(1)\subseteq W$, so we have a contradiction. 
    Let $c\in \WDec(1)$. Define
        \begin{equation*}
        	x_c = \{\langle\mathbf{0},\varnothing\rangle \mid 0\in c\}.
        \end{equation*}
    We will show that $V(\mathcal{A})\Vdash x_c\in a$. If $c=1$, then we can see that
    $\mathbf{p}(\bblambda f.\mathbf{p00})(\bblambda f.\mathbf{p00})\Vdash x_c=1$.
    Therefore $\lnot\lnot (c=1)$ implies $\lnot\lnot(\exists g: g\Vdash x_c=1)$, which is equivalent to $\mathbf{0}\Vdash \lnot\lnot(x_c=1)$. (In fact, we can replace $\mathbf{0}$ to any $e\in\mathcal{A}$.)
    It is easy to see that $\bblambda f.\mathbf{p00}\Vdash x_c\subseteq 1$. 
    Combining with \eqref{Formula:NotWDecPersistence00}, we have
	\begin{equation*}
		t\Vdash x_c\in a.
	\end{equation*}
	for $t=e\cdot \mathbf{p}(\bblambda f.\mathbf{p00})\mathbf{0}$. Hence, there is $z\in V(\mathcal{A})$ such that $\langle (t)_0,z\rangle\in a$ and $(t)_1\Vdash z=x_c$. We can see that
	\begin{align}
		0\in c & \iff \langle \mathbf{0},\varnothing\rangle \in x_c \\
		& \iff \text{$z$ is inhabited},
	\end{align}
	so $c=\{0\mid \text{$z$ is inhabited}\} \in W$.
\end{proof}

The following preservation results allow us to prove various compatibility results. 
For example, our preservation results prove the theorems in \cite{Hahanyan1981}%
\footnote{Note that Hahanyan \cite{Hahanyan1981} uses a different axiomatic system we have considered. However, we can translate the results of Hahanyan via methods in \cite[Chapter VII, \S1]{Beeson1985}.} and more: %
\begin{theorem}
    If $\mathsf{ZF}$ is consistent, then the following set of axioms is consistent with $T_0 = \mathsf{IZF + RDC + PAx}$ or $T_1=\mathsf{CZF+REA+RDC+PAx}$ respectively:
    \begin{enumerate}[label=(\alph*)]
        \item \label{Theorem:IZFDCom0}
        $\mathsf{DCom + CT_0 + MP + IP + UP}$
        \item \label{Theorem:IZFDCom2}
        $\mathsf{DCom + }$ There is a D-infinite set $D$ such that
        $D$ and ${^D}D$ are equipotent.
    \end{enumerate}
\end{theorem}
\begin{proof}\pushQED{\qed}
    We start from a model $V$ of $\mathsf{ZFC}$ for $T_0$, or the model constructed in \cite{Aczel1986} for $T_1$. Let us consider the realizability models $V(\mathcal{A})$.
    We know that $V(\mathcal{A})$ satisfies $\mathsf{IZF}$ or $\mathsf{CZF}$ if $V$ does. (See Chapter 3 of \cite{McCartyPhD} and \cite{Rathjen2003Realizability} for its proof.)
    Since $V$ satisfies $\mathsf{RDC}$ and $\mathsf{PAx}$, $V(\mathcal{A})$ also satisfy $\mathsf{RDC}$ and $\mathsf{PAx}$ by Theorem 3.1.12 of \cite{Dihoum2016}. In the case of $\mathsf{REA}$ over $\mathsf{CZF}$, apply Theorem 6.2 of \cite{Rathjen2003Realizability}.
	
	The consistency of $T_i$ ($i=0,1$) with \ref{Theorem:IZFDCom0} comes from the combination of \autoref{Theorem:PreservationDComOrig}, \cite[Chapter 3 and 4]{McCartyPhD}, ans \cite[Theorems 7.1 and 9.2]{Rathjen2003Realizability} under the Kleene realizability.
    \ref{Theorem:IZFDCom2} follows from \autoref{Theorem:PreservationDComOrig} and \cite[Theorem 4.1.1]{Dihoum2016} under $\mathcal{A}=D_\infty$.
\end{proof}
We may show the consistency of $\mathsf{IZF+DCom}$ with Brouwnian principles in the same vein by considering the realizability model over Kleene's second algebra $\mathcal{K}_2$. We may establish the details for that by applying facts on \cite{Dihoum2016} and \cite{Beeson1985}, but we will not make any progress in this article.

\section{Functional realizability}\label{Section: FunctionalRealizability}
In this section, we discuss why Hahanyan's model does not satisfy Powerset by showing it actually satisfies $\mathsf{ADCom}$. He followed an old-fashioned intensional realizability model to formulate his model, but we will not follow Hahanyan's original formulation. Instead, we follow the definition that is implicit in Swan's model $V^f_0(\mathcal{A})$ appearing in \cite{Swan2014}. Despite the different details, the basic idea behind the functional realizability is the same: Assigning a unique name to each realizer.

    \begin{definition}
        Let $\mathcal{A}$ be a pca and $V(\mathcal{A})$ be the realizability universe. A name $a\in V(\mathcal{A})$ is \emph{functional} if for every $\langle e,b\rangle, \langle e, c\rangle \in a$, $b=c$. $V^f(\mathcal{A})$ is the class of all functional names $a\in V(\mathcal{A})$.
        
        The realizability relation $\Vdash^f$ over $V^f(\mathcal{A})$ is identical with the usual realizablity relation $\Vdash$ over $V(\mathcal{A})$. (See \cite[Definition 4.1]{Rathjen2003Realizability}), except for that we restrict quantifiers to $V^f(\mathcal{A})$ instead of $V(\mathcal{A})$.
        We call $V^f(\mathcal{A})\models \phi(\vec{a})$ if there is $e\in\mathcal{A}$ such that $e\Vdash^f \phi(\vec{a})$.
    \end{definition}
    
    Let $\mathsf{ZFC^-}$ be $\mathsf{ZFC}$ without Powerset but with Collection instead of Replacement, and $\mathsf{GC}$ be the axiom of Global choice postulating there is a well-order of $V$.
    We can see that $\mathsf{ZFC^- + GC}$ proves $V^f(\mathcal{A})$ is a model of $\mathsf{CZF+Sep}$:
    \begin{theorem}[$\mathsf{ZFC^-+ GC}$]
        $V(\mathcal{A})\models \mathsf{CZF+Sep}$.
    \end{theorem}
    \begin{proof}
        The proof is identical to the proof of \cite[Theorem 4.3]{Swan2014} except for Strong Collection. For Strong Collection, assume that we have $e\Vdash^f \forall x\in a\exists y \phi(x,y)$.
        Then for each $\langle f,x\rangle \in a$ there is $y\in V^f(\mathcal{A})$ such that $ef \Vdash^f \phi(x,y)$. Now choose $y_f$ for each $f$ by applying Global Choice. Now take $c=\{\langle f,y_f\rangle \mid \langle f,x\rangle\in a\}$. Clearly, $c$ is functional. Furthermore, we can see that $c$ witnesses Strong Collection by following Swan's proof of \cite[Theorem 4.3]{Swan2014}.
    \end{proof}
	As a corollary, we can reproduce Lubarsky's proof-theoretic equivalence between $\mathsf{CZF+Sep}$ and Second-order Arithmetic:
	\begin{corollary}[Lubarsky \cite{Lubarsky2006SOA}]
	    Second-order Arithmetic interprets $\mathsf{CZF+Sep}$.
	\end{corollary}
	\begin{proof}
	    We proved that $\mathsf{ZFC^- + GC}$ interprets $\mathsf{CZF+Sep}$. Then the conclusion follows from a well-known fact that $\mathsf{ZFC^- + GC}$ and Second-order Arithmetic are mutually interpretable. 
	\end{proof}
	
	Hahanyan claimed that functional realizability would satisfy Powerset. However, it turns out that this is not true even if we work over $\mathsf{ZFC+GC}$.
	\begin{proposition}[$\mathsf{ZFC+GC}$]
	    Let $\mathcal{A}$ be a non-trivial pca. Then $V^f(\mathcal{A})$ thinks $\WDec(1)$ cannot be a subset of any set.
	\end{proposition}
	\begin{proof}
	    For $X$ a subset of $\mathcal{A}$, define 
	    \begin{equation*}
	        a_X = \{\langle f,\varnothing\rangle \mid f\in X\}.
	    \end{equation*}
	    We can see that if $X$ is inhabited then $\lambda f.\mathbf{p0i_r}\Vdash^f a_X\subseteq 1$. Furthermore, if $f\in X$, then $f\Vdash^f 0\in a_X$, and this shows $e\Vdash^f \lnot\lnot (0\in a_X)$ for any $e\in\mathcal{A}$.
	    Hence we can find a realizer $r$ uniform to $X$ such that $r\Vdash^f \lnot\lnot(a_X=1)$, so we can find another realizability $r'$ satisfying $r'\Vdash^f a_X\in \WDec(1)$ for all nonempty $X\subseteq\mathcal{A}$ by \autoref{Formula:WDec1}.
	    Now assume that $b\in V^f(\mathcal{A})$ and $e\in\mathcal{A}$ satisfy 
	    \begin{equation*}
	        e\Vdash^f \forall x [x\subseteq \WDec(1) \to x\in b].
	    \end{equation*}
	    Then for any inhabited $X\subseteq \mathcal{A}$, $e(\lambda f.\mathbf{p0i_r})\Vdash^f a_X\in b$.
	    Thus for each inhabited $X\subseteq \mathcal{A}$ we can find $\langle(e(\lambda f.\mathbf{p0i_r}))_0,c \rangle\in b$ such that $(e(\lambda f.\mathbf{p0i_r}))_1\Vdash^f a_X=c$.
	    By functionality of $b$, $c$ does not depend on $X$. Hence we can find a realizer $\alpha$, which does not depend on $X, Y\subseteq \mathcal{A}$, such that $\alpha\Vdash a_X=a_Y$.
	    
	    Now observe that $\alpha\Vdash a_X=a_Y$ implies $(\alpha)_0\Vdash a_X\subseteq a_Y$. By unpacking it, we have
	    \begin{equation*}
	        \forall \langle f,\varnothing\rangle \in a_X \exists \langle (\alpha f)_0, \varnothing\rangle \in a_Y [(\alpha f)_1 \Vdash^f \varnothing=\varnothing].
	    \end{equation*}
	    
	    Hence $f\mapsto (\alpha f)_0$ is a function from $X$ to $Y$, regardless of what $X$ and $Y$ are. This leads to a contradiction if we take, say, $X=\{\mathbf{k}\}$, $Y=\{\mathbf{k}\}$ and $Y=\{r\}$ for some $r\in\mathcal{A}$, $r\neq \mathbf{k}$. Then the map $f\mapsto (\alpha f)_0$ sends $\mathbf{k}$ to $\mathbf{k}$ and $r$, which is impossible.
	\end{proof}

\overfullrule=0pt
\printbibliography
\nocite{*}

\end{document}